\documentclass{article}

\usepackage{graphicx}
\usepackage[fleqn]{amsmath}
\usepackage{amsthm}
\usepackage{amssymb}
\usepackage{amsbsy}
\usepackage{amsfonts}
\usepackage{mathrsfs}
\usepackage[all]{xy}
\usepackage{amstext}
\usepackage{amscd}
\usepackage[dvips]{epsfig}
\usepackage{psfrag}
\usepackage{enumerate}
\usepackage{flafter}
\allowdisplaybreaks

\theoremstyle{plain}
\newtheorem{thm}{Theorem}[section]
\newtheorem{prop}[thm]{Proposition}
\newtheorem{cor}[thm]{Corollary}
\newtheorem{lem}[thm]{Lemma}
\theoremstyle{definition}
\newtheorem{exa}[thm]{Example}

\newtheorem{rem}[thm]{Remark}
\newtheorem{defn}[thm]{Definition}

\def\Ker{\mathop{\mathrm{Ker}}\nolimits}
\def\Coker{\mathop{\mathrm{Coker}}\nolimits}
\def\Hom{\mathop{\mathrm{Hom}}\nolimits}

\def\F{\mathop{\mathbb{F}_q}\nolimits}
\def\Fp{\mathop{\mathbb{F}_p}\nolimits}
\def\gr{\mathop{\mathrm{gr}}\nolimits}
\def\qq{\mathop{\mathfrak{q}}\nolimits}

\newcommand{\As}{{\rm As}}
\newcommand{\col}{{\rm Col}}

\newcommand{\ww}{{\omega}}
\newcommand{\lra}{\longrightarrow}
\newcommand{\ra}{\rightarrow}

\newcommand{\Z}{{\Bbb Z}}

\newcommand{\X}{\widetilde{X}}

\newcommand{\CC}{\mathcal{C}}
\newcommand{\sh}{\mathcal{S}}

\newcommand{\pc}[2]{\mbox{$\begin{array}{c}
\includegraphics[scale=#2]{#1.eps}
\end{array}$}}

\begin{document}
\large
\begin{center}
{\bf\Large On third homologies of groups and of quandles }
\end{center}

\vskip 0.5pc

\begin{center}
{\bf\Large 
via the Dijkgraaf-Witten invariant and Inoue-Kabaya map}
\end{center}

\vskip 1.5pc
\begin{center}
{\Large Takefumi Nosaka}
\end{center}
\vskip 1pc
\begin{abstract}\baselineskip=12pt \noindent We propose a simple method to produce quandle cocycles from group cocycles, as 
a modification of Inoue-Kabaya chain map.
We further show that, in respect to ``universal central extended quandles", the chain map induces an isomorphism between their third homologies.
For example, all Mochizuki's quandle 3-cocycles are shown to be derived from 
group cocycles of some non-abelian group. 
As an application, we calculate some $\Z$-equivariant parts of the Dijkgraaf-Witten invariants of some cyclic branched covering spaces, 
via some cocycle invariant of links. 
% using the group. We compute some Massey triple products via the former invariant.
\end{abstract}

\begin{center}
\end{center}

\large
\baselineskip=16pt 
\section{Introduction}A quandle, $X$, is a set with a binary operation whose definition 
was partially motivated from knot theory. 
Fenn-Rourke-Sanderson \cite{FRS1,FRS2} defined a space $BX$ called rack space, in analogy to the classifying spaces of groups. % homologies.
Furthermore, Carter {\it et. al} \cite{CJKLS,CKS} introduced
quandle cohomologies $H^*_Q(X;A)$ with local coefficients, by slightly modifying the cohomology of $BX$;
they further defined combinatorially
a state-sum invariant $I_{\psi}(L)$ of links $L $ constructed from a cocycle $\psi \in H^*_Q(X;A)$. 
The construction can be seen as an analogue
of the Dijkgraaf-Witten invariant \cite{DW} of closed oriented 3-manifolds $M$ constructed from a finite group $G$ and a 3-cocycle $\kappa \in H^3_{\gr}(G;A)$:
To be specific, the invariant is defined as the formal sum of pairings expressed by
\begin{equation}\label{DWDW} \mathrm{DW}_{\kappa }(M ):= \sum_{ f \in \Hom_{\mathrm{gr}} ( \pi_1(M), G ) } \langle f^{*}(\kappa ), [M] \rangle \in \Z[A], 
\end{equation}
where $[M ] $ is the fundamental class in $ H_3(M ;\Z)$. 
Inspired by this analogue, for many quandles $X$, the author \cite{Nos2} gave essentially topological meanings of the cocycle invariants 
with using the Dijkgraaf-Witten invariant and the homotopy group $\pi_2(BX)$. 

We mainly focus on a relation between quandle homology and group one. 
There are several such studies. For example, the second quandle homology is well-studied by Eisermann \cite{Eis2} from 
first group homologies. 
In addition, the author \cite{Nos2} roughly computed some third quandle homologies from the group homologies of $\pi_1(BX)$
with some ambiguity. 
%the computations were by the mean of Postnikov decomposition of $BX$, hence, were rough and brute. 
Furthermore, for any quandle $X$, Inoue-Kabaya \cite{IK} constructed a chain map $\varphi_{\rm IK}$ from the quandle complex to a certain complex. 
Although the latter complex seems far from something familiar, 
Kabaya \cite{K} modified the $\varphi_{\rm IK}$ mapping to a group homology under a certain strong condition of $X$. 
Furthermore, for certain special quandles, the author \cite{Nos3} proposed a method to construct quandle cocycles from invariant theory via the chain map.

This paper demonstrates a relation between third homologies of groups and those of quandles via the Inoue-Kabaya map, with respect to a broad subclass of quandles.
%s a simple way. 
%the $\varphi_{\rm IK}$ with respect to a broad? subclass of quandles, 
Here a quandle in the subclass is defined as a group $G$ with an operation 
$g \lhd h := \rho (gh^{-1}) h$ for $g,h \in G$, where $\rho : G \ra G$ is a fixed group isomorphism (Definition \ref{defq}). 
Denote such a quandle by $X = (G, \rho)$. 
Furthermore, denote by $H_n^{\gr }(G; \Z)_{\Z} $ a quotient of the group homology of $G$ subject to the action by the $\rho$, called the $\Z$-coinvariant. 
In \S \ref{SScon31}, we reformulate the Inoue-Kabaya map, $\Phi_n$, which induces a homomorphism 
$$(\Phi_n)_* : H_n^Q(X;\Z ) \lra H_n^{\gr }(G; \Z)_{\Z}. $$
Furthermore, we lift this map $\Phi_n$ to being
a chain map $\varphi_n$ from $C_n^Q(X;\Z ) $ to the usual group homology $H_n^{\gr }(G; \Z)$; see Proposition \ref{haihaihai}. 
%first construct a chain map $\varphi_n$ (Definition \ref{}) by a simple formula, which induces 
As a corollary, if found a presentation of a group $n$-cocycle $\kappa $ of $G$, we easily obtain that of the induced quandle $n$-cocycle $\varphi_n^*( \kappa )$. 
Hence, this approach is expected to be useful of computing the quandle cocycle invariant constructed from such a quandle $(G,\rho)$. 

%To verify applications of the map $(\varphi_n)_* $ in more details, we prepare some terminologies: and take the projection $\pi_{\rho}: H_n^{\gr }(G; \Z) \ra H_n^{\gr }(G; \Z)_{\Z} $. 
%Denote $ \Phi_n$ the composite $ \pi_{\rho} \circ (\varphi_n)_* $. 

This paper moreover investigates properties of the chain map $\Phi_n $ above. 
To begin with, we focus on a class of universal quandle coverings $\widetilde{Y}$, introduced by Eisermann \cite{Eis2}. 
Roughly speaking about this $\widetilde{Y} $, given a ``connected" quandle $Y$ of finite order, we can set the quandle $\widetilde{Y}$ and an epimorphism $p_Y : \widetilde{Y} \ra Y$ (as a quandle covering); 
further the quandle $\widetilde{Y} $ is of the form $(\Ker (\epsilon_Y), \rho) $ for some group $\Ker (\epsilon_Y )$; see Example \ref{ex3} for details. 
We then show that the associated chain map $\widetilde{\Phi}_3 $ induces an isomorphism
\begin{equation}\label{phi12} (\widetilde{\Phi}_3)_*: H_3^Q (\widetilde{Y}) \cong H_3^{\gr } ( \Ker (\epsilon_Y ))_{\Z} \ \ \ \ \ \ \mathrm{up \ to }\ t_Y \textrm{-torsion}, \end{equation}
where $t_Y \in \mathbb{N} $ is the minimal satisfying $\rho^{t_Y} = \mathrm{id}$ (Theorem \ref{thm13}).

Needless to say, the $\Phi_3$ is not always isomorphic for such quandles $(G,\rho)$; 
However, by the help of the universality of coverings, in some cases we can analyse the map $\Phi_3$ as follows. 

Next, we will demonstrate Mochizuki quandle 3-cocycles \cite{Moc2}, 
which are most known quandle 3-cocycles so far. 
Consider quandles $Y$ of the forms ($\F, \times \omega$) with $\omega \in \F$, called Alexander quandle usually, where we regard the finite field $\F$ as an additive group and the symbol $\times \omega$ is a $\omega$-multiple of $\F$. 
He found all 3-cocycles of $Y$ by solving a certain differential equations over $\F$ and his statement was 
a little complicated (see \S \ref{reMoc}). However, in this paper we easily obtain and explain his all 3-cocycles from some group 3-cocycles via the map $ \Phi^* _n$ 
(see \eqref{jijii} and Lemma \ref{bababa1}). 
Moreover, we show that the third quandle cohomology $H^3_Q(Y ;\F) $ 
is isomorphic to a sum of some group homologies via the maps $ \Phi_2 $, $ \Phi_3 $ and $ \widetilde{\Phi}_3$ (see Theorem \ref{thm2} in details). 
In conclusion, all the Mochizuki 3-cocycles stems from some group 3-cocycles
via the three maps.

Furthermore, we propose a relation to a partial sum of some Dijkgraaf-Witten invariants of $\widehat{C}_L^t $, 
where $\widehat{C}_L^t $ denotes the $t$-fold cyclic covering space of $S^3$ branched over a link $L$.
See \eqref{DWDWder2} for the detailed definition of the partial sum, and denote it by $\mathrm{DW}^{\Z}_{\kappa}( \widehat{C}_L^t ) \in \Z [A] $. 
To be specific, we show (Theorem \ref{thm3}) that if the induced map $p_Y^* :H_Q^3 (Y ;A)\ra H_Q^3 (\widetilde{Y} ;A)$ is surjective, and 
if $Y$ is connected and of finite order, 
then any group $3$-cocycle $\kappa$ of the above group $\Ker (\epsilon_Y)$ admits some
quandle 3-cocycle $\psi$ of $Y$ for which the equality 
$$ \mathrm{DW}^{\Z}_{\kappa}( \widehat{C}_L^t ) = I_{\psi}(L) \in \Z [A]$$
holds. Here the right side $I_{\psi}(L) $ is the quandle cocycle invariant of links $L$ \cite{CKS} 
(see Remark \ref{thm343t} for some quandles satisfying the assumption on $p_Y^* $). %, and is easy to compute. 
While the equivalence of the two invariants was implied in the previous paper \cite{Nos2} by abstract nonsense, 
the point is that the cocycle $\psi$ is definitely obtained from the chain map $\widetilde{\Phi}_3$. %our theorem says that, more concretely, the left $\mathrm{DW}^{\Z}_{\kappa}( \widehat{C}_L^t )$ can be computed. 

We here emphasize that our theorem serves as computing some parts of the Dijkgraaf-Witten invariants $\mathrm{DW}^{\Z}_{\kappa}( \widehat{C}_L^t )$
via the right invariant $ I_{\psi}(L) $. 
A known standard way to compute the invariant is to find a fundamental class from a triangulation of $M$ (see \cite{DW, Wakui}). 
However, presentations of group 3-cocycles are intricate in general. 
So, most known computations of the Dijkgraaf-Witten invariants are those with respect to abelian groups.
However, in computing them via the right invariant $I_{\psi}(L)$, we use no triangulation of $M$ and 
many quandle 3-cocycles are simpler than group ones (in our experience).

In fact, in \S \ref{k1ti2}, we succeed in some computations of the formal sums 
$\mathrm{DW}^{\Z}_{\kappa}( \widehat{C}_L^t )$ by using the Mochizuki 3-cocycles, which are derived from triple Massey products of a meta-abelian group $G_X$ (see Proposition \ref{lll}). 
For example, we will calculate the cocycle invariants of the torus knots $T(m,n)$ (see Theorem \ref{torus}); 
hence we obtain the partial sum $\mathrm{DW}^{\Z}_{\kappa} $ of the Brieskorn manifold $\Sigma (m,n,t )$, which is the covering space branched over the knot $T(m,n)$.
%non-trivial values of the Dijkgraaf-Witten invariants (Corollary \ref{torus314}). 
Furthermore, as a special case $\ww = -1$, we compute the cocycle invariant of some knots $K$,
and, hence, obtain some values $\mathrm{DW}^{\Z}_{\kappa}( \widehat{C}_K^t )$ for the double covering spaces branched along $K$ (see Table 1 in \S \ref{k41ti2}).

This paper is organized as follows. 
In \S 2, we introduce a lift of Inoue-Kabaya chain map and state theorems. 
In \S 3, we prove Theorems \ref{thm13} and \ref{thm3}.
In \S 4, we show that Mochizuki 3-cocycles are derived from some group 3-cocycles. 
%prepare some notation and observe properties to prove the theorems. The theorems are proven . 
In \S 5, we calculate some partial sum of the Dijkgraaf-Witten invariants. 

\vskip 0.57pc
\noindent
{\bf Notation and convention} 
A symbol 
$\F$ is a finite field of characteristic $p>0 $.
Denote $H_n^{\rm gr}(G) $ the group homology of a group $G$ with trivial integral coefficients. 
Furthermore we assume that manifolds are smooth, connected, oriented.
%Furthermore, any branched coverings of $S^3$ are cyclic, and 
%we denote by $\widehat{C}_{L}^{m}$ the $m$-fold cyclic covering of $S^3$ branched over a link $L$.

\section{Results}\label{assdas}

In \S \ref{results} and \ref{results2}, we state theorems.
For this, we briefly review quandle homologies and their properties in \S \ref{adsewbi}, 
and we modify the Inoue-Kabaya map in \ref{SScon31}.

\subsection{Review of quandles and quandle cohomologies }\label{adsewbi}

We start by recalling basic concepts about quandles.
A $quandle$, $X$, is a set with a binary operation $(x, y) \rightarrow x \lhd y$ such 
that, for any $x,y,z \in X,$ $ x \lhd x = x$, $(x \lhd y) \lhd z = (x \lhd z)\lhd(y\lhd z)$ and 
there exists uniquely $w \in X$ such that $w \lhd y = x$.
A quandle $X$ is said to be {\it of type $t_X$}, if $t_X >0$ is the minimal $N$ number satisfying 
$ a= (\cdots (a\lhd b) \cdots) \lhd b $ [$N$-times on the right with $b$] for any $a,b \in X$.
Furthermore, {\it the associated group} of $X$, $ \mathrm{As}(X) $, is defined to be the group presented by 
$$ \mathrm{As}(X) := \langle \ e_x \ (x \in X) \ | \ e_{x \lhd y}^{-1}e_y^{-1}e_x e_y \ \ \ \ \ (x, y \in X ) \ \rangle .$$
The group $\As (X)$ acts on $X$ by the formula $ x \cdot e_y := x \lhd y$ for $x,y \in X$. 
If the action is transitive, $X$ is said to be {\it connected}. 
Furthermore, take a homomorphism $ \epsilon_X: \As (X) \ra \Z $ sending $e_x$ to $1$;
so we have a group extension 
\begin{equation}\label{eseq} 0 \lra \Ker (\epsilon_X ) \stackrel{\iota}{\lra} \As (X) \stackrel{\epsilon_X}{\lra} \Z \lra 0\ \ \ (\mathrm{exact}). \end{equation}

Next, we introduce a subclass of quandles which we mainly use in this paper. 
\begin{defn}[{\cite[\S 4]{Joy}}]\label{defq}
Fix a group $G$ and a group isomorphism $ \rho : G \ra G$. 
Equip $X=G$ with a quandle operation by setting 
\begin{equation}\label{Gqdl} g \lhd h := \rho (g h^{-1})h . \end{equation}
\end{defn}

\noindent
Note that the quandle $(G,\rho )$ is of type $t_X$, if and only if the $t_X$-th power of $\rho$ is the identity, i.e., $\rho ^{t_X } = \mathrm{id}_G$. 
%Furthermore, $(G,\rho )$ is connected if the map $G \ra G $ sending $g$ to $g^{-1}\rho(g)$ is surjective. 
% For emxample, if $G$ is a connected affine algebraic group over $\F$ and $\rho$ is a Frobenius map, the map is surjective by Lang theorem. 

%For $G$ be a group. A {\it core quandle} (of $G$) is the $G$ with the binary operation $g \lhd h := h g^{-1}h$.
%This quandle is of type $2$. 
Although this class of such quandles $(G, \rho)$ is a subclass of quandles, it includes interesting quandles as follows: 
\begin{exa}[{Alexander quandle}]\label{ex1}
Let $X=G$ be an abelian group. Denoting $\rho$ by $T$ instead, 
we can regard $X$ as a $\Z[T^{\pm 1} ]$-module. Then the quandle operation is rewritten in 
$$ x \lhd y := Tx +(1-T)y , $$
called {\it Alexander quandle}. 
Given a finite field $\mathbb{F}_q$ and $\omega \in \mathbb{F}^{*}_q$ with $\omega \neq 1$,
the quandle of type $X=\F[T]/(T -\ww)$ is called {\it Alexander quandle on} $\F$ {\it with} $\ww$.

The type $t_X$ of $X$ equals the minimal $n $ satisfying $T^n =1 $ in $X$. 
We easily check that $X$ is connected if and only if $(1-T)$ is invertible.
\end{exa}

\begin{exa}[{Universal quandle covering}]\label{ex3}
Given a connected quandle $X $, consider the kernel $G=\Ker (\epsilon_X) $ in \eqref{eseq}. 
Fix $a \in X$. Using a group homomorphism $\rho_a: \Ker (\epsilon_X) \ra \Ker (\epsilon_X) $ define by $ \rho_a (g) = e_a^{-1}g e_a$, we have a quandle $\X =(\Ker (\epsilon_X) , \rho_a)$, called {\it extended quandle of $X$}. 
We easily see the independence of the choice of $a \in X$ up to quandle isomorphisms.

Considering the restriction of the action $X \curvearrowleft \As(X)$ to $\Ker (\epsilon_X)$, 
a map $p_X : \X \ra X$ sending $g$ to $a \cdot g$ is known to be a quandle homomorphism (see \cite[Theorem 4.1]{Joy}), 
and called {\it (universal quandle) covering} \cite{Eis2}. 
It can easily be seen that if $X$ is of type $t_X$ and of finite order, so is $\X$. 
Furthermore, the $\X $ is shown to be connected \cite[Lemma 6.8]{Nos2}), 
\end{exa}

Finally, we briefly review the quandle complexes introduced by \cite{CJKLS}.
Let $X$ be a quandle.
Let us construct a complex by putting the free $\Z$-module $ C_n^R(X) $ spanned by $( x_1, \dots , x_n) \in X^n$
and letting its boundary $\partial_n^{R}( x_1, \dots , x_n) \in C_{n-1}^{R }(X)$ be 
$$ \sum_{2 \leq i \leq n} (-1)^i\bigl(( x_1, \dots,x_{i-1},x_{i+1},\dots,x_n) - (x_1 \lhd x_i,\dots,x_{i-1}\lhd x_i,x_{i+1},\dots,x_n)\bigr).$$
The composite $\partial_{n-1}^R \circ \partial_n^R $ is zero. The pair $(C_*^R(X), \partial^R_*)$ is called {\it rack complex}.
Let $C^D_n (X)$ be a submodule of $C^R_n (X)$ generated by $n$-tuples $(x_1, \dots ,x_n)$
with $x_i = x_{i+1}$ for some $ i \in \{1, \dots , n-1\}$ if $n \geq 2$; otherwise, let $C_1^D (X)=0$.
Since $ \partial_n^R (C^D_n (X) ) \subset C^D_{n-1} (X)$, we can define a complex $\bigl( C^Q_* (X), \partial_* \bigr) $ by the quotient $C^R_n (X) /C^D_n (X)$.
The homology $H^Q_n (X) $ is called {\it quandle homology} of $X$.
Dually, we can define the cohomologies $H_R^n (X;A)$ and $H_Q^n (X;A) $ with a commutative ring $A$.

However, the second term of the differential $\partial_n^R$ seems incomprehensible from the definition. 
In the next subsection, for a quandle of the form $(G, \rho)$, we give a simple formula of the $\partial_n^R$. 

\subsection{A lift of Inoue-Kabaya chain map}\label{SScon31}
We now construct a chain map \eqref{vphi} with respect to a class of quandles in Definition \ref{defq}.
Our construction is a modification of Inoue-Kabaya map \cite[\S 3]{IK} (see Remark \ref{ikmap}).

In this subsection, %we fix a group $G$ and a group isomorphism $\rho:G \ra G$, and % to be either or the inverse map. 
we often denote $\rho (x)$ by $x^{\rho }$ and $\rho ^n (x)$ by $x^{n \rho }$ for short, respectively. 

For quandles $X$ of the forms $(G,\rho ) $ in Definition \ref{defq}, we will reformulate 
the rack complex $C_n^R (X) (\cong \Z \langle G^n \rangle)$ in non-homogeneous coordinates. 
Define another differential $\partial^{R_G}_n : C_n^R (X) \ra C_n^R (X)$ by setting
$$ \partial^{R_G}_n (g_1, \dots, g_n )\! := \!\!\!\! \sum_{ 1 \leq i \leq n-1}\!\!\!\!\! (-1)^i \bigl( (g_1, \dots, g_{i-1}, g_{i}g_{i+1}, g_{i+2}, \dots, g_n) - (g_1^{\rho }, \dots, g_{i-1}^{\rho}, g_{i}^{\rho}g_{i+1}, g_{i+2}, \dots, g_n)\bigr).  $$
We easily check $ \partial^{R_G}_{n-1} \circ \partial^{R_G}_n=0 $, and can further 
see 
\begin{lem}\label{exle1} Take a bijection $G^n \ra G^n $ defined by setting 
\begin{equation}\label{bijection} (x_1,\dots, x_n ) \longmapsto (x_1 x_2^{-1},x_2 x_3^{-1},\dots, x_{n-1} x_n^{-1},x_n ) . \end{equation}
This map yields a chain isomorphism $\Upsilon : (C_n^R (X), \partial^{R}_n)\cong (C_n^R (X), \partial^{R_G}_n)$. 
\end{lem}
\begin{proof}By direct calculation. \end{proof}
%that the isomorphism $ \Upsilon : C_n^R (X) \ra C_n^R (X)$ itself coming from a bijection is
Furthermore, we define a subcomplex $D_n(G)$ generated by $n$-tuples $(g_1, \dots, g_n)$ such that $g_i =1 $ for some $i \leq n-1$. 
We denote the quotient complex by $ C_n^{Q_G}(X)$. 
By Lemma \ref{exle1} this homology, $H_n^{Q_G}(X)$, is isomorphic to the quandle homology $H_n^{Q}(X)$ in \S \ref{adsewbi}.

We briefly review 
{\it normalized} chain complexes of groups, $C^{\rm gr}_n(G)$, in {\it non-}homogeneous terms (see, e.g. \cite{Bro}) as follows:
Let $ \overline{C}_{n}^{\rm gr}(G) $ denote the free $\Z$-module generated by $G^n$, and let its boundary map $\partial_n^{\mathrm{gr}}( g_1, \dots , g_n) \in \overline{C}_{n-1}^{\mathrm{gr}}(G)$ be 
\[ ( g_2, \dots ,g_{n}) +\sum_{1 \leq i \leq n-1}\! (-1)^i ( g_1, \dots ,g_{i-1}, g_{i} g_{i+1}, g_{i+2},\dots , g_n)+(-1)^{n} ( g_1, \dots , g_{n-1}) .\]
Furthermore, %consider the relation $ (g_1, \dots, g_n) =0 $ for some $ 1 \leq i < n$ and $g_i =1$, and denote it by $C_{n}^{\rm gr}(G)$. 
concerning the submodule $ D_n(G) $ mentioned above, we easily check $ \partial_n^{\mathrm{gr}} \bigl( D_n(G) \bigr) \subset D_{n-1}(G)$.
We denote by $C_{n}^{\rm gr}(G) $ the quotient complex of $ \overline{C}_{n}^{\rm gr}(G) $ modulo $ D_n(G) $.
As is well-known, this homology coincides with the usual group homology of $G$ (see \cite[\S I.5]{Bro}).
% $H_n^{\rm gr}(G)$ which we have used.

We next construct a chain map $\varphi_n $ from the complex $C_{n}^{R_G}( X)$ to 
another $ C_{n}^{\rm gr}(G) $. 
\begin{defn}\label{defIKlift}
Assume that a quandle $X$ of the form $(G, \rho)$ is of type $t_X $. Take a set 
$$\mathcal{K}_n:= \{ (k_1, \dots, k_{n} ) \in \mathbb{Z}^n \ | \ 0 \leq k_i - k_{i-1} \leq 1 , \ \ 0 \leq k_n \leq t_X -1 \ \}.$$
of order $ t_X 2^{n-1}$. 
%For $ \Bbbk_{n} \in \mathcal{K}_n$, let $| \Bbbk_{n}| \in \Z $ be $ k_1+ \cdots + k_{n}$.,
We define a homomorphism $\varphi_n : C_{n}^{R_G}( X) \ra C_{n}^{\rm gr}(G) $ by setting %the expression
$$ \varphi_n (g_1,g_2, \dots, g_n ) = \sum_{ (k_1, \dots, k_n)\in \mathcal{K}_n } \! (-1)^{k_1 } ( g_{1}^{k_1 \rho }, g_{2}^{k_2 \rho }, \dots, g_{n}^{k_n \rho}) \in C_n^{\rm gr}(G ). $$

\end{defn}

For example, when $n=3$, the $\varphi_3 (x,y,z)$ is written in 
\begin{equation}\label{vphi} \sum_{ 0 \leq i \leq t_X -1 } (x^{i \rho} , y^{i \rho}, z^{i \rho} ) - (x^{(i+1) \rho} , y^{i \rho}, z^{i \rho} ) - (x^{(i+1) \rho} , y^{(i+1) \rho}, z^{i \rho} ) + (x^{(i+2) \rho} , y^{(i+1) \rho}, z^{i \rho} ). \end{equation}

%Here, if $k_j =0$, then we define $U_{j} \cdot (g_2 \cdots g_{k_j +1} ) =U_j $.
%For example, when $n= 2$ and $n=3 $, the definition of the map $\varphi_n $ is rewritten in 
%\begin{equation}\label{agd2} \ \ \varphi_2(f,g; a, b )= (a, b ) -(a \cdot g , b ) , \end{equation}
%\begin{equation}\label{agd} \!\!\!\!\!\!\! \varphi_3 (f,g,h; a, b,c ) = (a, b,c ) -(a \cdot g , b ,c) - (a \cdot h , b \cdot h ,c ) + (a \cdot (gh) , b \cdot h ,c ).\end{equation}
\begin{prop}\label{haihaihai} Let $X$ be a quandle of the form $(G, \rho)$. 
If $X$ is of type $t_X < \infty $, then the homomorphism $\varphi_n \! :\! C_{n}^{R}( X) \ra C_n^{\mathrm{gr}}(G)$ is a chain map. Namely,
$\partial_n^{\mathrm{gr}} \circ \varphi_n = \varphi_{n-1} \circ \partial_n^{R_G}$. 

Furthermore the image of $D_n(G)$ is zero. In particular, the $\varphi_n$ induces a chain map from the quotient $ C_n^{Q_G}(X)$ to $ C_n^{\mathrm{gr}}(G) $, 
and a homomorphism $H_n^{Q_G}(X) \ra H_n^{\gr }(G )$. 

\end{prop}
\begin{proof} The identity $\partial_n^{\mathrm{gr}} \circ \varphi_n = \varphi_{n-1} \circ \partial_n^{R_G}$ can be proven by direct calculation similar to \cite[Lemma 3.1]{IK} or \cite[Appendix]{Nos3}, so we omit the details. 
It is not hard to check the latter part directly. \end{proof}
Accordingly, we obtain an easy method to quandle cocycles from group cocycles: 

\begin{cor}\label{to14}
Let a quandle $X= (G, \rho)$ be of type $t_X$. 
For a normalized group $n$-cocycle $\kappa $ of $G$, then the pullback $\varphi_n^*(\kappa)$ is a quandle $n$-cocycle. 
\end{cor}

\begin{rem}\label{ikmap}
We now roughly compare our map $\varphi_n $ with a chain map $\varphi_{\rm IK } $ introduced by Inoue and Kabaya \cite{IK}.
For any quandle $Q$, they constructed a complex ``$C_n^{\Delta}(Q)$" from a simplicial object, and formulated the map $\varphi_{\rm IK }: C_{n}^R( Q) \ra C_n^{\Delta}(Q)$ in its {\it homogeneous} coordinate system (see \cite[\S 3]{IK} for details).

To see this in some detail, we define a module, $ C_n^{\mathrm{gr}}(G) _{\Z}$, to be the quotient of $C_n^{\mathrm{gr}}(G) $ modulo the relation $ (g_1, \dots, g_n ) =(\rho (g_1), \dots, \rho(g_n))$, called $\Z${\it -coinvariant of $ C_n^{\mathrm{gr}}(G)$}. 
We denote by $\pi_\rho $ the projection $ C_n^{\mathrm{gr}}(G)\ra C_n^{\mathrm{gr}}(G) _{\Z}$. 
We can see that, if $Q$ is a quandle of the form $(G, \rho)$ and connected, then the above complex $C_n^{\Delta}(Q )$ is isomorphic to the coinvariant $ C_n^{\mathrm{gr}}(G) _{\Z}$; 
further, we can check the equality $t_X \cdot \varphi_{\rm IK} = \pi_\rho \circ \varphi_n $. 
In summary, our map $\varphi_n$ is of a lift of the Inoue-Kabaya map $\varphi_{\rm IK } $ in connected cases, and is relatively simple. 
So we fix a notation: \end{rem}
\begin{defn}\label{def2}
Let $\Phi_n $ denote the composite chain map $ \pi_\rho \circ \varphi_n : C_{n}^{Q_G }( X) \ra C_n^{\mathrm{gr}}(G) _{\Z}$, i.e., 
$$ \Phi_n : C_n^{Q_G}(X) \stackrel{ \varphi_n }{ \lra} C_n^{\rm gr } (G) \xrightarrow{ \ \rm proj \ } C_n^{\rm gr } (G)_{\Z }. $$
\end{defn}

Incidentally, we prepare a `reduced map' of the $\Phi_n $, which is used temporarily in Theorem \ref{thm3}.
Consider a homomorphism $\mathcal{P}: C_n^{R} (X) \ra C_{n-1}^{R_G} (X) $ derived from a map 
$X^n \ra X^{n-1}$ sending $(x_1,\dots, x_n ) $ to $(x_1,\dots, x_{n-1} ) $. 
We discuss the composite $\Phi_{n-1} \circ \mathcal{P}$ as follows:
%$\pi_\rho$
\begin{prop}\label{aip} Let $X$ be a quandle $(G, \rho)$ of type $t_X$.
The composite $\Phi_{n-1} \circ \mathcal{P}: C_{n}^{R_G}( X) \ra C_{n-1}^{\mathrm{gr}}(G)_{\Z}$ is a chain map. 
Furthermore it induces a chain map from the quotient $ C_n^{Q_G}(X)$ to $ C_{n-1}^{\mathrm{gr}}(G)_{\Z} $.
\end{prop}
\begin{proof} By direct calculation (cf. Proposition \ref{haihaihai} and \cite[Lemma 3.1]{IK}).\end{proof}
%Finally, to summarise this section, we describe the two chain homomorphisms: 

\subsection{Results on the chain map $\Phi_3$}\label{results}
In this paper, we study the chain map $\Phi_n $ with $n=3$ (Theorems \ref{thm13}, \ref{thm2}). 

We first study the maps $\Phi_n $ with respect to extended quandles in Example \ref{ex3}. 
\begin{thm}\label{thm13}
Let $X$ be a connected quandle of type $t_X $,
and $\X=(\Ker (\epsilon_X), \rho_a )$ be the extended quandle in Example \ref{ex3}. 
Let $\widetilde{\Phi}_n$ denote the chain map in Definition \ref{def2}.
%Put the inclusion $\iota: \Ker (\epsilon_X ) \ra \As (X)$ in \eqref{eseq}.
Assume that the $H_3^{\gr}(\As(X))$ is finitely generated, e.g., $X$ is of finite order. 
Then the induced map 
$$ (\widetilde{\Phi}_3)_* : H_3^Q(\X) \lra H_3^{\rm gr}(\Ker(\epsilon_X))_{\Z} $$
is an isomorphism modulo $t_X $-torsion.
%In particular, the $(\varphi_3)_* : H_3^Q(\X) \ra H_3^{\rm gr}(\Ker (\epsilon_X))$ is an injection without $m$-torsion.
\end{thm}
\begin{rem}\label{thm13rem}
We here compare this theorem with the result \cite[Theorem 3.18]{Nos2} which stated an existence of 
an isomorphism $ H_3^Q(\X) \cong H_3^{\rm gr}(\As (X))$ modulo $t_X$.
So this theorem says that the chain map $(\widetilde{\Phi}_3)_* $ gives an explicit presentation of this isomorphism. 
Indeed we late get a canonical isomorphism $ H_3^{\rm gr}(\As (X)) \cong H_3^{\rm gr}(\Ker(\epsilon_X))_{\Z} $ modulo $t_X$; see Lemma \ref{bap}. 
\end{rem}

Next, as a special case, we focus on the Alexander quandles on $\F$ in Example \ref{ex1}.
Using the maps $\Phi_n$, we will characterize the third quandle cohomology from group homologies.
Identifying $X= \F$ with $(\Z_p)^h$ as an additive group, let $\rho : \F \ra \F$ be the $\omega$-multiple. 
We then have a chain map $ \Phi_n^* :C^n_{\rm gr}( (\Z_p)^h )^{\Z} \ra C^n_Q(X)$, 
and later show the following:

\begin{prop}\label{prop2}
Let $X$ be an Alexander quandle on $\F$ with $\ww$ in Example \ref{ex1}. 
Then the induced map $\Phi_3^* :H^3_{\rm gr}( (\Z_p)^h ;\F )^{\Z} \ra H^3_Q(X;\F ) $ is injective. 

Furthermore, if $H_2^Q(X)$ vanishes, then this $ \Phi_3^* $ is an isomorphism.
\end{prop}
In general, this $ \Phi_3^* $ is not surjective. 
To solve the obstruction $H_2^Q(X) $, %we consider 'the central extension by the $H_2^Q(X)$' of the $ \Phi_3^* $: precisely, 
we consider the chain map $\widetilde{\Phi}_n: C_n^Q(\X) \ra C_n^{\rm gr}(\Ker (\epsilon_X))_{\Z} $ with respect to the extended quandle (Example \ref{ex3}). 
%putting the extendd quandle $\X$, we denote $ $ by 
We then obtain a commutative diagram
$$ \! \! \xymatrix{
& H^n_{\gr} ((\Z_p)^h ;\F )^{\Z } \ar@{->}[rrr]^{ \ \ \Phi_n^*} \ar[d]^{\mathrm{Proj}^*}& & & \ H^n_{Q} ( X ;\F) \ar[d]^{p_X^*}& \\
& H^n_{\gr} ( \Ker (\epsilon_X) ;\F )^{\Z} \ar[rrr]^{ \ \ \widetilde{\Phi}_n^*} & & & H^n_{Q} (\X ;\F ). \\
}$$
Remark that, when $n=3$, the bottom map $\widetilde{\Phi}_3^* $ is an isomorphism by Theorem \ref{thm13}.
%by using Theorem \ref{thm13} and a transfer $ H^3_{\gr} ( \Ker (\epsilon_X) ;\F )^{\Z} \ra H^3_{\gr} ( \As (X) ;\F )$.
Denote $\mathrm{res}( \widetilde{\Phi}_3^* )$ the isomorphism restricted on the cokernel $\Coker (\mathrm{Proj}^*)$. 
In addition, we take the chain map $\Phi_{n-1} \circ \mathcal{P} : C_{n}^{Q_G}( X) \ra C_{n-1}^{\mathrm{gr}}(G)_{\Z} $ in Proposition \ref{aip}, 
% denote it by $\underline{\varphi_{n}}$.

To summarize these homomorphisms, we characterize the third quandle cohomology of $X$:

\begin{thm}\label{thm2}
Let $X$ be an Alexander quandle on $\F$. 
Let $q$ be odd. 
%Then the map $ \varphi_3^* $ is injective. 
Then there is a section $\mathfrak{s} : H^3_Q( \X ;\F )\ra H^3_Q( X ;\F ) $ of $p_X^*$ such that 
the following homomorphism is an isomorphism: 
\begin{equation}\label{ac9} \notag ( \Phi_{2} \circ \mathcal{P})^* \oplus \Phi_3^* \oplus \bigl( \mathfrak{s} \circ \mathrm{res}( \widetilde{\Phi}_3^* )\bigr) : \end{equation}
\begin{equation}\label{ac} \ \ \ \ \ \ H^2_{\rm gr}( (\Z_p)^h ;\F )^{\Z} \oplus H^3_{\rm gr} ( (\Z_p)^h ;\F )^{\Z} \oplus \Coker (\mathrm{Proj}^*) \lra H^3_Q( X;\F ). \end{equation}
\end{thm}
\noindent The proof will appear in \S \ref{reviewsw}. 
In conclusion, all the Mochizuki 3-cocycles are derived from group 3-cocycles of $(\Z_p)^h$ and of $\Ker (\epsilon_X) $
via the chain map $\Phi_{n}$.

\

Incidentally, in higher degree, we now observe that the induced map $(\varphi_n)_* : H_{n}^{Q}( X) \ra H_n^{\mathrm{gr}}(G)$ is far from injective and surjective. 
\begin{exa}\label{thm2rem}
To see this, letting $q=p $, we observe the chain map $\varphi_n $ with respect to an Alexander quandle $X$ on $\mathbb{F}_p$ in Example \ref{ex3}. 
Then, we easily see $\Ker (\epsilon_X) \cong \Z_p$ (cf. \eqref{lower}). The cohomology 
$H^n_{\gr}(\Z_p ;\mathbb{F}_p)$ is $\mathbb{F}_p$ for any $n \in \mathbb{N}$. 
On the other hand, in \cite{Nosa}, the integral quandle homology $ H_n^{Q}(X) $ was shown to be 
$(\Z_p)^{b_n} $, where $b_n \in \Z$ is determined by the recurrence formula 
\[b_{n+2t }=b_{n}+b_{n+1}+b_{n+2}, \ \ \ b_1= b_2= \cdots = b_{2t -2}=0, \ \ \mathrm{and} \ \ b_{2t -1}=b_{2t}=1,\]
and $t>0$ is the minimal number satisfying $\omega^t =1$.
In conclusion, since the $b_n$ is an exponentially growing, the map $(\varphi_n)_*$ is not bijective.
\end{exa}

\subsection{Shadow cocycle invariant and Dijkgraaf-Witten invariant}\label{results2}
Furthermore, we address a relation between 
shadow cocycle invariant \cite{CKS} and the Dijkgraaf-Witten invariant \cite{DW}.
We now review the both invariants, and state Theorem \ref{thm3}.

%the chain map $\varphi_3$ with respect to core quandles on groups $G$ (Example \ref{}).

First, to describe the former invariant, we begin reviewing $X$-colorings. 
Given a quandle $X$, an $X$-{\it coloring} of an oriented link diagram $D$ 
is a map $\CC: \{ \mbox{arcs of $D$} \} \to X$ 
satisfying the condition in the left of Figure \ref{fig.color} at each crossing of $D$. 
Denote by $\col_X(D)$ the set of $X$-colorings of $D$. % ,$ \col_X(D)$ has been much-studied. 
Note that two diagrams $D_1$ and $D_2$ related by Reidemeister moves admit a 1:1-correspondence $\col_X(D_1) \leftrightarrow \col_X(D_2)$; see \cite{CJKLS,CKS} for details.

We further define a {\it shadow coloring} to be a pair of an $X$-coloring $\CC $ and a map $\lambda $ from the complementary regions of $D$ to $X$ such that, 
if regions $R$ and $R ' $ are separated by an arc $\alpha$ as shown in the right of Figure \ref{fig.color}, the equality $\lambda(R ) \lhd \CC (\alpha )= \lambda (R ')$ holds.
Let $\overline{\col}_{X}(D )$ denote the set of shadow colorings of $D$.
Given an $X$-coloring $\CC $, we put $x_0 \in X$ on the region containing a point at infinity.
Then, by the rules in Figure \ref{fig.color}, colors of other regions are uniquely determined, and ensure a shadow coloring $\sh $ denoted by $(\CC; x_0)$.
We thus obtain a bijection $ \col_X (D) \times X \simeq \overline{\col}_X(D)$ sending $(\CC,x_0)$ to $\sh= (\CC; x_0)$.

\begin{figure}[htpb]
\begin{center}
\begin{picture}(100,70)
\put(-122,52){\Large $\alpha $}
\put(-66,51){\Large $\beta $}
\put(-66,18){\Large $\gamma $}
\put(-112,33){\pc{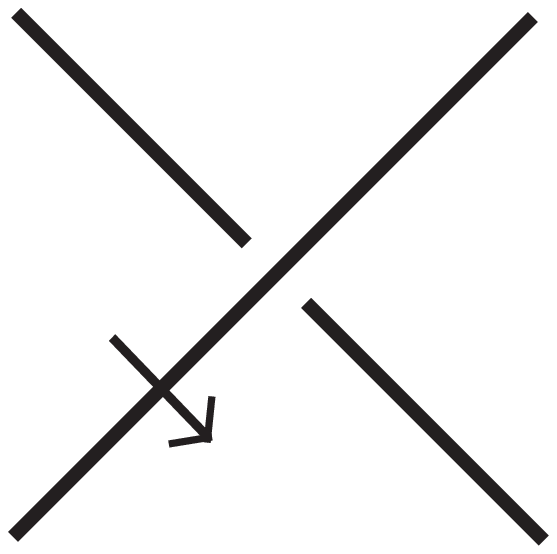}{0.2530174}}

\put(63,33){\pc{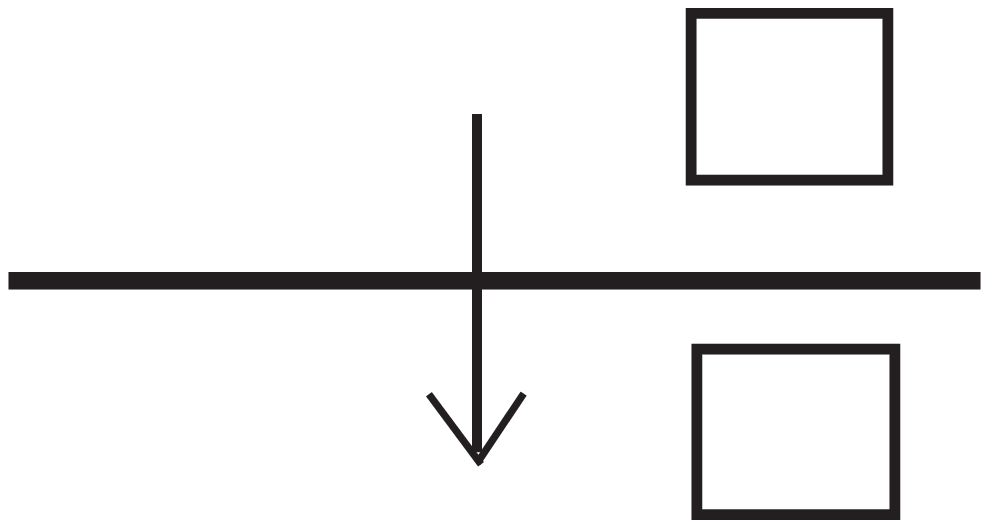}{0.266}}
\put(-50,16){$C(\alpha) \lhd C(\beta) = C(\gamma)$}

\put(88,43){\large $ \alpha $}
\put(123,46){\large $ R $}
\put(123,21){\large $ R ' $}
\put(164,35){\large $ \lambda(R) \lhd \CC (\alpha )= \lambda(R') $}
\end{picture}
\end{center}
\vskip -1.9937pc
\caption{The coloring conditions at each crossing and around arcs.
\label{fig.color}}
\end{figure}

We briefly formulate (shadow) quandle cocycle invariants \cite{CKS}. 
Let $D$ be a diagram of a link $L$, and $\sh \in \overline{ \col}_X(D) $ a shadow coloring.
For a crossing $\tau$ shown in Figure \ref{fig.color43}, 
we define a {\it weight of} $\tau $ to be $\epsilon_{\tau} (x,y,z) \in C_3^Q(X;\Z)$, where $\epsilon_{\tau} \in \{ \pm 1\}$ is the sign of $\tau$.
Further {\it the fundamental class of $\sh$} is defined to be $[\sh]:= \sum_{\tau}\epsilon_{\tau} (x,y,z) \in C^Q_3(X;\Z) $, and is known to be a 3-cycle.
For a quandle 3-cocycle $\psi \in C^3_Q(X;A) $, we consider the pairing $\langle \psi, [\sh ]\rangle \in A $.
%, where the pairing is done after projecting the rack complex to the quandle one.
The formal sum $I_{\psi}(L):=\sum_{\sh \in \overline{\col}_X(D)} 1_{\Z } \{ \langle \psi, [\sh ]\rangle \} $ in the group ring $\Z[A]$ is called {\it quandle cocycle invariant of} $L$,
where a symbol $1_{\Z} \{a \} \in \Z[A]$ means the basis represented by $a \in A$.
%If $\psi$ is null-cohomologous, $I_{\psi}(L)$ lies in $\Z$.
By construction, in order to calculate the invariant concretely, it is important to find explicit formulas of quandle 3-cocycles.

%Incidentally, for a 2-cocycle $\phi \in C^2_Q(X;A) $, a map $ \widetilde{\phi}: X^3 \ra A$ sending $(x,y,z)$ to $\phi(y,z)$ is a 3-cocycle, and 
%the invariant $\Phi_{\widetilde{\phi}}(L)$ coincides with the original cocycle invariant in \cite{CJKLS}. 

\vskip -1.3937pc
\begin{figure}[htpb]
\begin{center}
\begin{picture}(100,80)
\put(-81,35){\Large $x $}
\put(-53,63){\Large $y $}
\put(-34,62){\Large $z $}
\put(-90,33){\pc{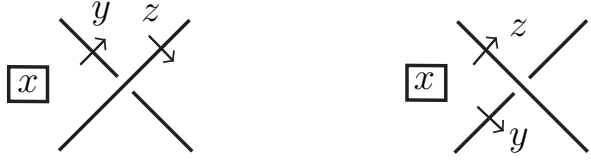}{0.3285}}

\put(69,36){\Large $x $}
\put(105,56){\Large $z $}
\put(105,15){\Large $y $}
\end{picture}
\end{center}
\vskip -1.737pc
\caption{Positive and negative crossings with $X$-colors.
\label{fig.color43}}
\end{figure}

On the other hand, we will briefly formulate a Dijkgraaf-Witten invariant below \eqref{DWDWder2}. 
For this, for a link $L$, denote by $\widehat{C}_{L}^{m} $ the $m$-fold cyclic covering space of $S^3$ branched over $L$.
Note that $\Z$ canonically acts on the space $\widehat{C}_{L}^{m} $ by the covering transformations. 
According to \cite{Nos2}, when $X$ is connected and of type $t$, 
for an $X$-coloring of $L$, we can construct a $\Z$-equivariant homomorphism $\Gamma_{\CC}: \pi_1( \widehat{C}_{L}^{t}) \ra \Ker(\epsilon_X) $, 
where $\Z$ acts on $\Ker(\epsilon_X) $ via the split surjection \eqref{eseq}; 
see \S \ref{adasjunbi} for the definition of $\Gamma_{\CC} $.
In summary, given a link-diagram $D$, we have a map 
\begin{equation}\label{lower23} \Gamma_{\bullet }: \col_X(D) \lra \Hom_{\rm gr}^{\Z}(\pi_1( \widehat{C}_{L}^{t}), \Ker(\epsilon_X)) , \end{equation} 
where the right side is the set of the $\Z$-equivariant group homomorphisms $\pi_1( \widehat{C}_{L}^{t}) \ra \Ker(\epsilon_X)$.
% where we denote by $\widehat{C}_{L}^{t}$ the $m$-fold cyclic covering of $S^3$ branched over the link $L$, 
%Let $X$ be a connected quandle of type $m$. Given an $X$-coloring $\CC \in \col_X(D)$, 
Furthermore, consider the pushforward of the fundamental class $ [ \widehat{C}_{L}^{t} ] \in H_3( \widehat{C}_{L}^{t})$
via the homomorphism $ \Gamma_{\mathcal{C}}$. 
Using this, with respect to a $\Z$-invariant 3-cocycle $\kappa $ of $\Ker (\epsilon_X)$, 
we define {\it a $\Z$-equivariant part of Dijkgraaf-Witten invariant of $ \widehat{C}_{L}^{t} $} by the formula 
\begin{equation}\label{DWDWder2} \mathrm{ DW}^{\Z}_{ \kappa }( \widehat{C}_{L}^{t} ) = \sum_{\mathcal{C} \in \col_X(D)} \langle \ \kappa, \ (\Gamma_{\mathcal{C}})_* ([\widehat{C}_{L}^{t}]) \ \rangle \in \Z[ A ]. 
\end{equation}

We will show that, with respect to a connected quandle with a certain assumption, 
the two invariants explained above are equivalent (see \S \ref{a241das} for its proof). Precisely,
\begin{thm}\label{thm3}
Let $X$ be a finite connected quandle of type $t_X$. 
Let an abelian group $A$ contain no $t_X$-torsion. 
Assume that the induced map $ p_X^* : H^3_Q(X ;A ) \ra H^3_Q(\X ;A)$ is surjective. 
Then any $\Z$-invariant 3-cocycle $\kappa $ of $\Ker (\epsilon_X)$ admits 
a quandle $3$-cocycle $\psi$ of $X$ and the equality 
$$ I_{\psi}(L) = |X| \cdot \mathrm{ DW}^{\Z}_{ \kappa }( \widehat{C}_{L}^{t_X} ) \in \Z [A]. $$
Moreover, conversely, given a quandle $3$-cocycle $\psi$ of $X$, 
there is a $\Z$-invariant group 3-cocycle $\kappa $ of $\Ker (\epsilon_X)$ for which the equality holds. 
\end{thm} 
\begin{rem}\label{thm343t}
As is seen in the proof in \S \ref{SScon34}, for some quandles, 
we can obtain the quandle cocycle $\psi$ in Theorem \ref{thm3} concretely from a group 3-cocycle $\kappa$. 
For instance, 
if $p_X : \X \ra X$ is isomorphic, then the $\psi$ is given by $ \varphi^*_3(\kappa)$.
As another example, for Alexander quandles on $\F$, the relations between 
$\psi$ and $\kappa$ are given by explicit formulas (see \S \ref{1305} in details).
\end{rem} 

To conclude, under the assumption on the $ p_X^* $, 
the invariant $\mathrm{ DW}^{\Z}_{ \kappa }( \widehat{C}_{L}^{t} )$ constructed from 
any $\Z$-invariant 3-cocycle $\kappa $ of $\Ker (\epsilon_X)$ is can be computed from the quandle cocycle invariants via link-diagrams. 
Fortunately, there are some quandles with the assumption of the surjectivity of $p_X^ *$. % : H^3_Q(X ;A ) \ra H^3_Q(\X ;A)$ is surjective. 
For example, connected Alexander quandles $X$ which satisfy either oddness of $|X|$ or evenness of $t_X$ (\cite[Lemma 9.15]{Nos2}), 
and ``symplectic quandles $X$ over $\F$" with $g=1$ (\cite[\S 3.3]{Nos2}).
%Also, we can find such quandles, by reading the list of some third quandle homologies in .

However, on the contrary, other quandles do not satisfy the assumption, e.g., with respect to symplectic quandles $X$ with $g>1$ in a stable range, which are of type $p$, as is shown in \cite[\S 3]{Nos2}, the $H_3^{\gr }(\Ker (\epsilon_X)) \cong H_3^{\gr }(Sp(2g ;\F)) \cong \Z/q^2 -1$ and $H_3^Q(X)\cong 0$.
Hence, in general, the invariant $ \mathrm{ DW}^{\Z}_{ \kappa }( \widehat{C}_{L}^{t_X} ) $ is not always interpreted from shadow cocycle invariants.

\section{Proofs of Theorems \ref{thm13} and \ref{thm3}, }\label{a241das}
Our objectivity in this section is to prove Theorems \ref{thm13} and \ref{thm3}. 
This outline is, roughly speaking, a translation from some homotopical results in \cite{Nos2} to terms of the group homology $H_3^{\gr}(\Ker (\epsilon_X))$: 
Actually, using homotopy groups, the author studied the group homology and a relation to Dijkgraaf-Witten invariant. 
So Section \ref{adasjunbi} reviews a group, $\Pi_2(X)$, and two homomorphisms from the group.
In \S \ref{SScon34}, we prove the theorems using a key lemma as such a translation.
In \S \ref{SScon345}, we give a proof of the key lemma. 
Readers who are interested in only Theorem \ref{thm2} may skip to \S \ref{reviewsw}. 

\subsection{Review of two homomorphisms $ \Delta_{X,x_0}$ and $\Theta_X $}\label{adasjunbi}

In order to construct the two homomorphisms in \eqref{san16}, \eqref{san1}, we first review the group $\Pi_2(X) $ defined in \cite{FRS1,FRS2}. 
Consider the set of all $X$-colorings of all link-diagrams.
Then we define $\Pi_2(X)$ to be the quotient set subject to Reidemeister moves and 
concordance relations illustrated in Figure \ref{pic.21...}. 
Then disjoint unions of $X$-colorings make $\Pi_2(X)$ into an abelian group.
For a connected quandle $X$ of finite order, the group $\Pi_2(X)$ was well-studied (see also Theorem \ref{nos3th} below). 

\begin{figure}[h]
\vskip -0.7pc
\vskip -0.17pc
$$
\begin{picture}(220,66)
\put(-70,22){\pc{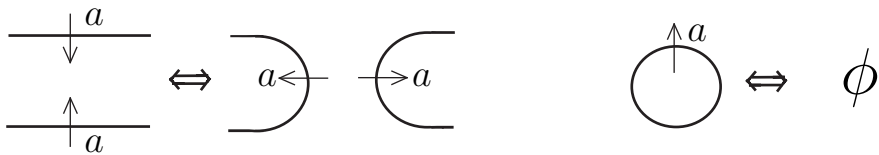}{0.506}}
\put(-35,47){\Large $a$}
\put(-35,-1){\Large $a$}
\put(89,24){\Large $a$}
\put(30.4,24 ){\Large $a$}
\put(193.2,41){\Large $a$}
\put(251,21){\Huge $\phi$}
\end{picture}$$

\vskip -0.7pc
\caption{\label{pic.21...}
The concordance relations }
\end{figure}

% Mochizuki \cite{Moc2} determined a basis of $H_Q^2(X ;\F) \oplus H_Q^3(X ;\F)$, when $X$ is an Alexander quandle on $\F$ with $\ww$ 
% (see also \S \ref{reMoc} for details of the 2-, 3-cocycles).

Next, we explain the homomrphism in \eqref{san16} below. 
Recall from \S \ref{results2} that, given an $X$-colirng $\CC$ and $x_0 \in X$, we can construct 
a shadow coloring of the form $(\CC; x_0)$, 
and the fundamental class $ [ ( \mathcal{C} ;x_0) ]$ contained in $H_3^Q(X)$. 
We easily see that, if two $X$-colorings 
$\mathcal{C}$, $\mathcal{C}'$ are
related by Reidemeister moves and concordance relations, 
then the associated classes $[(\mathcal{C} ;x_0)]$, $[(\mathcal{C}' ;x_0)] $ are equal in $ H_3^{Q }(X ) $ by definition. 
Hence we obtain a homomorphism 
\begin{equation}\label{san16} \ \Delta_{X,x_0} : \Pi_2(X) \lra H_3^{Q }(X ), \ \ \ \ \ \mathcal{C} \longmapsto [(\mathcal{C} ;x_0)] . \end{equation}

On the other hand, we will explain the homomorphism $\Theta_{X}$ below \eqref{san1}. 
For this end, we first observe the fundamental group of the $t$-fold cyclic covering space $\widehat{C}_{L}^{t} $.
%$\pi_1( \widetilde{S^3 \setminus L})$. 
Given a link-diagram $D$ of $L$, let $\gamma_0 , \dots, \gamma_n $ be the arcs of $D$. 
Consider Wirtinger presentation of $\pi_1( S^3 \setminus L)$ generated by $\gamma_0 , \dots, \gamma_n $.
For $s \in \Z$, we take a copy $\gamma_{i, s}$ of the arc $\gamma_i$.
Then, by Reidemeister-Schreier method (see, e.g., \cite[\S 3]{K} for the details),
the group $\pi_1(\widehat{C}_{L}^{t})$ can be presented by 
\[ \mathrm{generators} : \ \gamma_{i,s} \ \ \ (0 \leq i \leq n, \ s \in \Z), \]
\[ \mathrm{relations } : \ \gamma_{k,s}= \gamma_{j,s-1}^{-1} \gamma_{i,s-1}\gamma_{j,s} \ \mathrm{for \ each \ crossings \ in \ the \ figure\ below , \ and \ }\]
\[\ \ \ \ \ \ \ \ \ \ \ \ \ \ \ \gamma_{i,s}= \gamma_{i,s+t}, \ \ \ \gamma_{0,s}=1. \]

%Further the inclusion $\iota: \pi_1(\widetilde{S^3 \setminus L}) \hookrightarrow \pi_1(S^3 \setminus L) $ is given by $\iota (\gamma_{i,s})= \gamma_0^{s-1} \gamma_i \gamma_0^{-s}.$
%Moreover, is obtained from the presentation of $\pi_1( \widetilde{S^3 \setminus L} )$ by adding relations 
\vskip 0.15pc 
\begin{figure}[htpb]
\begin{center}
\begin{picture}(100,60)
\put(-66,50){\large $\gamma_i $}
\put(-15,50){\large $\gamma_j $}
\put(-15,22){\large $\gamma_k $}

\put(-66,33){\pc{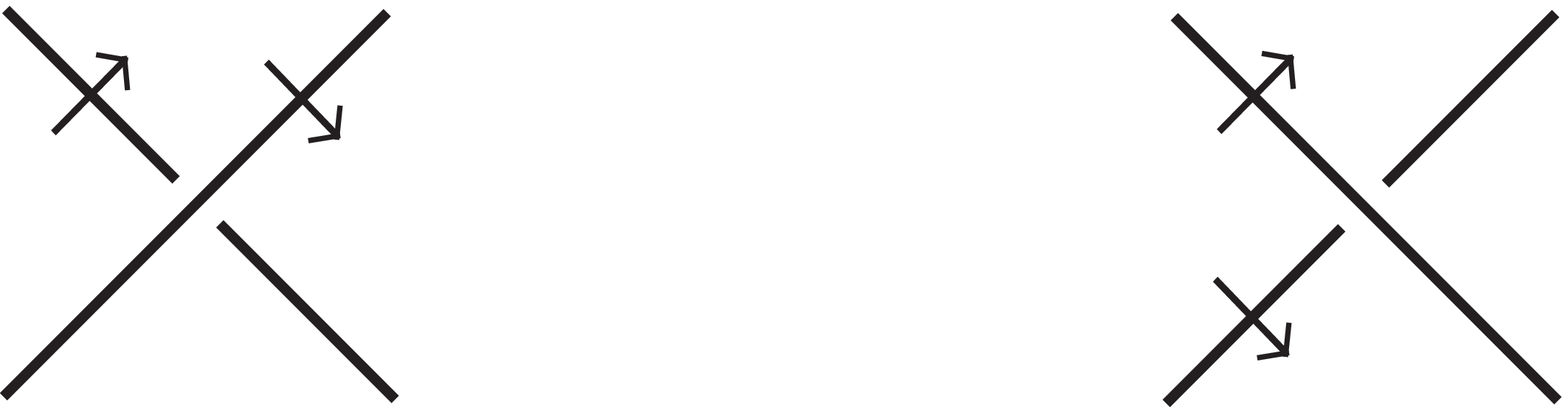}{0.314}}

\put(78,50){\large $\gamma_j $}
\put(132,50){\large $\gamma_k $}
\put(78,22){\large $\gamma_i $}
\end{picture}
\end{center}

\vskip -1.7pc

\vskip -0.5pc 
%\caption{\label{koutenpn} Positive and negative crossings.}
\end{figure}

Let $X$ be a connected quandle of type $t$.
Given an $X$-coloring $\CC \in \col_X(D)$, 
we denote the color on the arc $\gamma_i$ by $x_i \in X$.
Define a group homomorphism $\Gamma_{\CC}: \pi_1(\widehat{C}_{L}^{t}) \ra \Ker (\epsilon_X)$ by setting
$ \Gamma_{\CC}(\gamma_{i,s}):= e_{x_0}^{s-1} e_{x_i} e_{x_0}^{-s}$ 
(see \cite[\S 4]{Nos2} for the well-definedness). Furthermore, consider
the pushforward $ (\Gamma_{\CC})_* ( [ \widehat{C}_{L}^{t}]) \in H_3^{\gr }(\Ker (\epsilon_X) )$, where 
$[\widehat{C}_{L}^{t}] $ is the fundamental class in $ H_3( \widehat{C}_{L}^{t})$.
We thus obtain a map 
$$ \theta_{X,D}: \col_X(D) \lra H_3^{\gr }(\Ker (\epsilon_X) ), \ \ \ \ \ \ \ \ \ \ \ \CC \longmapsto (\Gamma_{\CC})_* ( [ \widehat{C}_{L}^{t}]). $$
As is shown \cite{Nos2}, if two $X$-colorings $\CC, \CC' $ can be related by Reidemeister moves and the concordance relations, 
then the identity $ \theta_{X,D}(\CC) = \theta_{X,D'}(\CC') $ holds. 
Therefore the maps $ \theta_{X,D} $ with respect to all diagrams $D$ yield a homomorphism
\begin{equation}\label{san1} \Theta_{X} : \ \Pi_2( X) \lra H_3^{\gr }(\Ker (\epsilon_X) ) . \end{equation}

This $\Theta_X$ plays an important role to study the group $\Pi_2(X)$ up to $t$-torsion. Indeed, 

\begin{thm}[{\cite[Theorems 3.4 and 3.18]{Nos2}}\footnote{In \cite{Nos2}, the maps $\Theta_{X}$, $\iota_* \circ \Theta_{X}$ were denoted by $ \Theta_{\Pi \Omega}$, $\Theta_{X} $, respectively.}]\label{nos3th} Let $X$ be a connected quandle of type $t_X$. 
Put the inclusion $\iota: \Ker (\epsilon_X ) \ra \As (X)$ in \eqref{eseq}.
If the homology $H_3^{\gr }(\As(X))$ is finitely generated, then the composite $\iota_* \circ \Theta_{X} : \ \Pi_2(X) \ra H_3^{\gr }(\As (X) ) $ is 
a split surjection modulo $t_X $-torsion , whose kernel is isomorphic to $H_2^Q(X)$ modulo $t_X $-torsion.

Furthermore, the induced map $(p_X)_*: \Pi_2(\X) \ra \Pi_2(X)$ is a split injection modulo $t_X$-torsion, and this cokernel equals the kernel of the composite $\iota_* \circ \Theta_{X} $. 
\end{thm}

\subsection{A key lemma and Proofs of Theorems \ref{thm13} and \ref{thm3}}\label{SScon34}

For the proofs, we state a key lemma. 
We here fix a terminology: A (shadow) $\X$-coloring of $D$ is said to be {\it based}, if an arc $\gamma_0$ of $D$ is colored by 
the identity $1_{\Ker (\epsilon_X)} \in \X$.

\begin{lem}[{cf. \cite[Theorem 9.1]{K}}\footnote{Kabaya \cite{K} showed a similar statement under a certain strong condition of quandles. 
However, as seen in the proofs of Theorems \ref{thm13} and \ref{thm3}, 
in order to verify a relation to the Dijkgraaf-Witten invariant, we may deal with only the extended quandle $\X $ in Example \ref{ex3}, which may not satisfy the strong condition. }]\label{lem11}
Let $X $ be a connected quandle of type $t_X < \infty $, 
and $p_X :\X \ra X$ the projection in Example \ref{ex3}. 
Let $\Upsilon: C_3^R(\X) \ra C_3^R(\X) $ be the chain isomorphism in \eqref{bijection}. 
Let $\sh  \in \overline{\col}_{\X }(D)$ be a based shadow coloring.
Define an $\X$-coloring $\widetilde{\CC} \in \col_{\widetilde{X}} (D)$ to be the $X$-coloring as the restriction of $\sh \in \overline{\col}_{\X }(D)$. 
Then 
$$ \Theta_X ( [p_X ( \widetilde{\CC})] ) = \varphi_{3} \circ \Upsilon ([\sh ]) \in H_3^{\gr }(\Ker (\epsilon_X)). $$
%In particular, if $|X| < \infty$, by considering the sum running over all $X$-colorings, we have 
\end{lem}

Before going to the proof, we will complete the proofs of Theorems \ref{thm13} and \ref{thm3}.

\begin{proof}[Proof of Theorem \ref{thm13}]
%By assumption, the homology $H_3^{\gr }(\As(X))$ is finitely generated. 
As mentioned in Remark \ref{thm13rem}, there is an isomorphism $H_3^Q(\X ) \cong H_3^{\gr }(\As(X) ) $ up to $t_X$-torsion, as finitely generated $\Z$-modules.
Hence, in order to prove that the map $(\widetilde{\Phi}_3)_* $ is an isomorphism, it suffices to show the surjectivity. 

To this end, we set the composite of the three homomorphisms mentioned above: 
$$ \Pi_2(\X) \xrightarrow{\ (p_X)_* \ } \Pi_2 ( X) \xrightarrow{\ \ \Theta_{X}\ \ } H_3^{\gr }(\Ker (\epsilon_{X}) ) \xrightarrow{\ \iota_* \ } H_3^{\gr }(\As(X) ). $$
It follows from Theorem \ref{nos3th} above that the composite is an isomorphism up to $t_X$-torsion.

Note that, Lemma \ref{bap} below ensures an isomorphism $\xi: H_3^{\gr}(\As(X)) \ra H_3^{\gr}(\Ker (\epsilon_X))_{\Z}$ 
such that $ \xi \circ \iota_* = t_X \cdot (\pi_{\rho})_*$, where $ \pi_{\rho}$ is the projection $ C_3^{\gr}(\Ker (\epsilon_X)) \ra C_3^{\gr}(\Ker (\epsilon_X))_{\Z}$ explained in \S \ref{SScon31}. 
Hence, the following composite is an isomorphism up to $t_X $-torsion as well: 
$$ (\pi_{\rho})_* \circ \Theta_{X} \circ (p_X)_* : \Pi_2(\X) \lra H_3^{\gr }(\Ker (\epsilon_{X}) )_{\Z}. $$

Therefore, for any 3-cycle $ \mathcal{K} \in H_3^{\gr }(\Ker (\epsilon_{X}) )_{\Z} $ which is annihilated by $t_X $, 
we choose some based $\X$-coloring $\widetilde{\mathcal{C}}$ such that $ \mathcal{K} =(\pi_{\rho})_* \circ \Theta_{X} \circ (p_X)_* ([\widetilde{\mathcal{C}} ])$. 
We here set a shadow coloring $\mathcal{S}$ of the form $(\widetilde{\mathcal{C}} ; 1_{\X})$.
Then, by the lemma \ref{lem11}, we notice the equalities 
$$ \widetilde{\Phi}_3( \Upsilon ([\mathcal{S}])) = (\pi_{\rho})_* \circ (\varphi_3 \circ \Upsilon ) ([\mathcal{S}]) = (\pi_{\rho})_* \circ \Theta_X ( p_X ([\widetilde{\mathcal{C}} ])) = \mathcal{K}. $$
Noting the left element $\Upsilon ([\mathcal{S}]) \in H^{Q_G}_3(\X)$, we obtain the surjectivity of $\widetilde{\Phi}_3 $ as required. 
\end{proof}%[Proof of Lemma \ref{SScon34}]

\begin{proof}[Proof of Theorem \ref{thm3}]
We first construct two homomorphisms \eqref{1111}, \eqref{1212}.
Let $X$ be a finite connected quandle of type $t_X$. 
Given a $\Z$-invariant group 3-cocycle $\kappa$, consider a composite homomorphism from $\Pi_2 (X)$: 
\begin{equation}\label{1111} \Pi_2 (X) \xrightarrow{\ \ \Theta_X \ } H_3^{\gr }(\Ker (\epsilon_X) ) \xrightarrow{\ \langle \kappa , \bullet \rangle \ } A. 
\end{equation}
%Here note that the kernel of $\Theta_X$ equals $H_2^{Q}(X) $ up to $m$-torsion.
\noindent
On the other hand, by the assumption of the surjectivity of $p_X^* : H^3_Q(X;A) \ra H^3_Q(\X;A)$, 
we can choose a quandle cocycle $\psi \in H^3_Q(X;A) $ such that $p_X^* (\psi) = (\varphi_3 \circ \Upsilon )^* ( \kappa )$. 
We then set a composite homomorphism 
\begin{equation}\label{1212} \Pi_2 (X) \xrightarrow{\ [ \bullet ;x_0 ] \ \ } H_3^{Q}(X) \xrightarrow{\ \langle \psi , \bullet \rangle \ } A. 
\end{equation}
We remark that this kernel contains the kernel of $\Theta_X $ by Theorem \ref{lem11}, since $A$ contains no $t_X$-torsion by assumption. 
%is shown \cite[Lemma ??]{Nos2}. 
%\footnote{The reason is roughly as follows: The deconmposition $ H_3^R(X) \cong \Z \oplus H_3^{Q}(X) \oplus H_2^{Q}(X) $ is known.
%The map $\Delta_{X,x_0 } $ factors thorugh $ H_3^R(X)$, however misses the last summand $H_2^{Q}(X)$.}).
%Actually, 

Next, we claim the equivalence of the two maps \eqref{1111} and \eqref{1212}.
%We denote the maps by $f_1$ and $f_2$, respectively. 
For this, we put some $\X$-colorings $\widetilde{\CC}_1, \dots, \widetilde{\CC}_n$, which generate the $\Pi_2(\X)$; here we may 
assume that these colorings are based by Lemma \ref{lembasic} below. 
Notice that, by Theorem \ref{nos3th}, the group $\Pi_2 (X)$ is generated by the kernel $\Ker (\iota_* \circ \Theta_X)$ and the elements 
$ p_X(\widetilde{\CC}_1 ), \dots, p_X(\widetilde{\CC}_n )$. 
%notice that $\Pi_2(X) \cong p_*( \Pi_2(\X)) \oplus H_2^Q(X)$ and the $p_*$ is injective.
Therefore, the claimed equivalence results from the following equalities:
%y $ f_1( p_*( \widetilde{\CC}_i ) )=f_2 (p_*( \widetilde{\CC}_i )) $. 
%Indeed, we have the equality by computation
$$ \langle \kappa , \Theta_X ( p_X ( \widetilde{\CC}_i ) \rangle  = \langle \kappa , \varphi_3 \circ \Upsilon ( [(\widetilde{\CC}_i; \tilde{x}_0) ] ) \rangle= \langle p_X^* (\psi) , [(\widetilde{\CC}_i; \tilde{x}_0) ] \rangle = \langle \psi , [p_X(\widetilde{\CC}_i); p_X(\tilde{x}_0)) ] \rangle , $$ 
where the first equality is obtained from Lemma \ref{lem11}.

We will show the equivalence of the two invariants as stated in Theorem \ref{thm3}. 
By definitions, these invariants are reformulated as 
$$\mathrm{DW}^{\Z}_{\kappa} ( \widehat{C}_{L}^{t_X } )= \sum_{\CC \in \col_X(D )} 1_{\Z } \{ \langle \kappa , \Theta_X (\CC) \rangle\}, \ \ \ \ I_{\psi} (L)= \sum_{x \in X} \sum_{\CC \in \col_X(D )} 1_{\Z } \{ \langle \psi , [\CC;x ] )\} \in \Z [A]. $$
Further, it is shown \cite[Theorem 4.3]{IK} that the right invariant $I_{\psi} (L)= |X| \sum_{\CC \in \col_X(D )} 1_{\Z } \{ \langle \psi , [\CC;x_0 ] )\}$ for any $x_0 \in X$.
In conclusion, since the homomorphisms \eqref{1111}, \eqref{1212} are equal as claimed above, so are the two invariants. 

Finally, to prove the latter part of Theorem \ref{thm3}, recall that the map $\widetilde{\Phi}_3$ is an isomorphism after tensoring $A$ (Theorem \ref{thm13}). 
So, given a quandle $3$-cocycle $\psi$, we 
define a $\Z$-invariant group 3-cocycle $\kappa $ of $\Ker (\epsilon_X)$ to be $ ( \Upsilon \circ \widetilde{\Phi}_3^* )^{-1} (p_X^* ( \psi )).$
Hence, by a similar argument as above, we have the desired equality $ I_{\psi}(L) = |X| \cdot \mathrm{ DW}^{\Z}_{ \kappa }( \widehat{C}_{L}^{t_X} ) $. \end{proof}

\subsection{Proof of the key lemma}\label{SScon345}

We will prove Lemma \ref{lem11} as a modification of \cite[Theorem 9.1]{K}. 
%is based on some results of \cite{IK}. % applying to for our group $G_X$.
%So we now review the discussions in \cite[\S 3 and 4]{K} in some details.

We will review descriptions in \cite[\S 4]{K}, in order to formulate concretely the orientation class $[\widehat{C}_{L}^{t}] \in H_3( \widehat{C}_{L}^{t};\Z )$ of
the branched covering space $\widehat{C}_{L}^{t} $, 
Let $N(L) \subset S^3$ denote a tubular neighborhood of $S^3 \setminus L$. 
Let $\gamma_0 , \dots, \gamma_n$ be the oriented arcs of the diagram $D$ such as \S \ref{adasjunbi}.
We may assume that each arc $\gamma_i$ has boundaries. 
For each arc $\gamma_i$, we can construct $4$ tetrahedra $T_{i}^{(u)} \subset \widehat{C}_{L}^{t} $ with $ 1 \leq u \leq 4$, and further decompose $ S^3 \setminus N(L) $ into these $4(n+1)$ tetrahedra as a triangulation,
which is commonly referred as to ``a standard triangulation" (see, e.g., \cite{Wee} or \cite[\S 4]{K}). 
Furthermore, 
Kabaya concretely constructed $4t $ tetrahedra $T_{i,s}^{(u)}$ in the branched covering space $ \widehat{C}_{L}^{t} $, where $0 \leq s \leq t -1$ and $1 \leq u \leq 4$ (see Figures 7 and 14 in \cite{K}).
Roughly speaking, the tetrahedron $T_{i,s}^{(u)}$ corresponds with the $s$-th connected component 
of the preimage of the $T_{i}^{(u)} \subset S^3 $ via the branched covering $ \widehat{C}_{L}^{t} \ra S^3$.
Let us fix the orderings of $T_{i,s}^{(u)}$ following Figure 8 in \cite{K}. 
Then he showed the union $\bigcup_{i,s ,u } T_{i,s}^{(u)} = \widehat{C}_{L}^{t} $ and that 
the formal sum $ \sum_{i ,s } \epsilon_{i} (T_{i,s}^{(1)} -T_{i,s}^{(2)} - T_{i,s}^{(3)}+ T_{i,s}^{(4)})$ represents the orientation class $[\widehat{C}_{L}^{t}]$, where $\epsilon_i \in \{ \pm 1\}$ is the sign of the crossing at the endpoint of $\gamma_i$.

We next discuss labelings of the tetrahedron $T_{i,s}^{(u)}$ by a group $G$. 
Put a map $\mathcal{L} : \{ T_{i,s}^{(u)} \}_{i,s,u} \ra G^3$.
Let $\mathcal{I} : \{ T_{i,s}^{(u)} \}_{i,s,u} \ra G$ be a constant map taking elements to the the identity of $G $.
We would regard the product $\mathcal{I} \times \mathcal{L} $ as a labeling of vertices in $T_{i,s}^{(u)}$s
according to the ordering. 
%We fix $i$, and consider a union $ \cup_{s,u} T_{i,s}^{(u)} \subset \widehat{C}_{L}^{l} $.
Fix a homomorphism $f : \pi_1( \widehat{C}_{L}^{t} )\ra G $ and recall the generators $\gamma_{i,s} \in \pi_1( \widehat{C}_{L}^{t} )$ in \S \ref{adasjunbi}.
He showed \cite[\S 4]{K} that, if
the labeling is globally compatible and the 
$\mathcal{L} \bigl( T_{i,s}^{(u)} \bigr) $s satisfy the following condition:
\begin{equation}\label{ha5f3ha} \bullet \ \mathcal{L} ( T_{i,s}^{(1)} ) \cdot f ( \gamma_{i,s}) = \mathcal{L}( T_{i, s+1}^{(3)}) , \ \ \ \ \ \ \ \ \ \mathcal{L} ( T_{i,s}^{(2)}) \cdot f ( \gamma_{i,s}) = \mathcal{L} ( T_{i, s+1}^{(4)} ) \in G^3, \ \ \ \
\end{equation} 
where the action $\cdot $ is diagonal, 
then the product $\mathcal{I} \times \mathcal{L}$ satisfies a group 1-cocycle condition in 
the union $\bigcup_{i,s ,u }  T_{i,s}^{(u)}$, and that the induced homomorphism coincides with the $f$. 
In the sequel, the pushforward $ f_*([\widehat{C}_{L}^{t}]) $ is represented by the formula 
\begin{equation}\label{repl2s} \Upsilon \bigl( \sum_{ 0 \leq i \leq n}\sum_{ 0 \leq s < t} \mathcal{L} \bigl(\epsilon_{i} (T_{i,s}^{(1)} -T_{i,s}^{(2)} - T_{i,s}^{(3)}+ T_{i,s}^{(4)})\bigr) \bigr) \in C_3^{\rm gr}(G), \end{equation}
where $\Upsilon: C_n^{\gr}(G) \ra C_n^{\gr}(G)$ is an isomorphism given in Lemma \ref{exle1}.

\begin{proof}[Proof of Lemma \ref{SScon34}]
The proof will be shown by expressing the left side $ \Theta_{X}( \CC )$ in details;

For this, 
given a shadow coloring $ \sh $ by $\X = \Ker (\epsilon_X)$, 
we will define a labelling $\mathcal{L} $ compatible 
with the homomorphism $\Gamma_{\CC}: \pi_1( \widehat{C}_{L}^{t} )\ra G $ with $G= \As (X)$ as follows.
Let $ (g,h,k) \in \X^3 = \Ker (\epsilon_X )^3$ be the weight of the endpoint of the arc $\gamma_i$.
Using the quandle structure on $\X $, we then define a map $\mathcal{L} : \{ T_{i,s}^{(u)} \}_{s,i,u} \ra (\Ker (\epsilon_X))^3 =\X ^3 $ by 
\[ \mathcal{L} ( T_{i,s}^{(1)} ) := ( g_{s-1},h_{s-1}, k_{s-1} ), \ \ \ \ \ \ \ \ \ \ \ \ \ \ \ \ \mathcal{L} ( T_{i,s}^{(2)} ) := (g_{s-1} \lhd h_{s-1} , h_{s-1}, k_{s-1} ), \]
\[ \mathcal{L} ( T_{i,s}^{(3)} ) := (g_s \lhd k_s ,h_s \lhd k_s, k_s ), \ \ \ \ \ \ \ \ \ \ \ \mathcal{L} ( T_{i,s}^{(4 )} ) := ((g_s \lhd h_s ) \lhd k_s ,h_s \lhd k_s, k_s ), \]
where we put $g_s :=e_a^{s} g e_a^{-s} \in\Ker (\epsilon_X)$ for short.

% We next check the condition \eqref{ha5f3ha} with respect to $\Gamma_{\CC}: \pi_1(S^3 \setminus L) \ra \Ker (\epsilon_X)$ in \eqref{}. 
We will verify the equalities \eqref{ha5f3ha} on $\mathcal{L}$. 
From the definition of the action $X \curvearrowleft \As (X)$, we notice 
an equality $ e_{p_X (k )}=e_{a \cdot k } = k^{-1}e_a k \in \As (X)$ for any $k \in \X$. 
In addition, we note $(p_X)_* ( \sh(\gamma_0)) = p_X( 1_{\X}) = a \in X $ since the $\mathcal{S}$ is based by assumption.
Hence, using the notation $\gamma_{i,s} \in \pi_1(\widehat{C}^t_L) $, we have 
\begin{equation}\label{k1kp2} \Gamma_{\CC}( \gamma_{i,s})= (e_a)^{s-1} e_{a \cdot \sh(\gamma_i) } e_a^{-s} = (e_a)^{s-1} e_{p_X (k )} e_a^{-s} = e_a^{s-1 } k^{-1}e_a k e_a^{-s}.
\end{equation}
%Furthermore, for any $b\in \Ker (\epsilon_X)$, we notice $ b \lhd^{1-s}e_a = e_a^{s-1} b e_a^{-s} \in \Ker (\epsilon_X), $ and 
%$$ (b \lhd k) \lhd^{-s}e_a = e_a^s (b \lhd k ) e_a^{1-s}= e_a^{s-1} b k^{-1} e_a k e_a^{-s} \in \Ker (\epsilon_X), $$
Therefore, for any $b\in X$, using \eqref{k1kp2}, we have the equality
$$ ( e_a^{s-1 } b e_a^{1-s } ) \cdot \Gamma_{\CC}( \gamma_{i,s}) = e_a^{s } ( b \lhd k) e_a^{-s } \in \Ker (\epsilon_X) . $$
Consequently, applying $b= g$, $b=h$ or $b=g \lhd h$ to this identity concludes the condition \eqref{ha5f3ha}.
Furthermore it is not hard to see that the labelling is globally compatible with all the triangulations. %, for any $i$. 
Hence, the labeling $ \mathcal{L}$ induces the homomorphism $\Gamma_{ \CC}: \pi_1( \widehat{C}_{L}^{t} )\ra \As (X) $.

Finally, we discuss the push-forward of the orientation class $(\Gamma_{ \CC})_* ([ \widehat{C}_{L}^{t}])\in C_3^{\gr}(\Ker (\epsilon_X);\Z).$
We first check that, for $x,y, z \in \X =\Ker (\epsilon_X) $, the following equality holds: 
\[ \Upsilon^{-1} \circ \varphi_3 \circ \Upsilon (x,y,z)= \] 
\[ \sum_{1 \leq s \leq t } (x_s ,y_s ,z_s)- (x_s \lhd y_s ,y_s,z_s)- (x_s \lhd z_s ,y_s\lhd z_s ,z_s)+ ((x_s \lhd y_s) \lhd z_s ,y_s\lhd z_s ,z_s).  \]
Here this verification is easily obtained from recalling the definitions of $\varphi_{3}$ in Definition \ref{defIKlift} and $ \Upsilon$ in \eqref{bijection}. 
%by a direct calculation we can check the equality
% $\Upsilon^{-1} \circ \varphi_3(xy^{-1}, yz^{-1},z)= \]
%\[ = \Upsilon^{-1} \bigl( (xy^{-1}, yz^{-1},z)- (\rho( xy^{-1}), yz^{-1},z)-(\rho( xy^{-1}), \rho(y z^{-1}),z)+ (\rho^2( xy^{-1}), \rho(y z^{-1}),z) \bigr) \]
%\[ = (x,y,z)- (x \lhd y ,y,z)- (x \lhd z ,y\lhd z ,z)+ ((x \lhd y) \lhd z ,y\lhd z ,z)\] 
Hence, compared with the map $\mathcal{L}$, we have the equality 
$$ \Upsilon^{-1} \circ \varphi_{3} \circ \Upsilon ([\mathcal{S}])= \sum_{i}\sum_{s } \mathcal{L} \bigl(\epsilon_{i} (T_{i,s}^{(1)} -T_{i,s}^{(2)} - T_{i,s}^{(3)}+ T_{i,s}^{(4)})\bigr) \in C_3^{\rm gr}(\Ker(\epsilon_X))$$ 
exactly. Since the right side is the push-forward $\Upsilon^{-1} \bigl( (\Gamma_{ \CC} )_* ([ \widehat{C}_{L}^{t}]) \bigr) $ by \eqref{repl2s},
we conclude the desired equality. \end{proof}

% \subsection{Proof of Theorem \ref{thm1}}\label{ssprthm2}
We now provide proofs of two lemmas above.

\begin{lem}\label{lembasic}
Let $X $ be a connected quandle.
If an element in $\Pi_2(\X)$ is represented by an $\X$-coloring of $D$, then 
the class equals some based $\X$-coloring of the $D$ in $\Pi_2(\X)$ .
\end{lem}
\begin{proof}
Let the arc $\gamma_0$ be colored by $ h \in \X$. 
Since the extended quandle $\X$ is also connected \cite[Lemma 9.15]{Nos2}, we have $g_1, \dots, g_n \in \X$ such that 
$ ( \cdots ( h \lhd g_1 ) \lhd \cdots ) \lhd g_n =1_{\X}$. 
Then by observing the following picture, we can change the $\CC$ to another $\X$-coloring $\CC'$ of $D$ such that 
the arc $\gamma_0$ is colored by $h \lhd g_{1}$ and that $ [\CC]= [\CC'] \in \Pi_2(\X)$. 

\vskip -0.639937pc
\begin{figure}[htpb]
\begin{center}
\begin{picture}(130,70)

\put(-142,52){ $\gamma_0 $}
\put(-117,54){\large $h $}
\put(-142,33){\pc{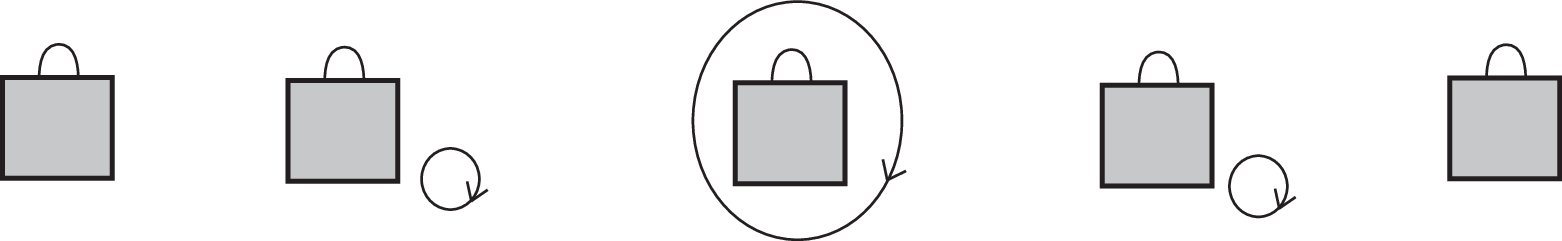}{0.5830174}}
\put(-94,32){\Large $=$}
\put(8,32){\Large $=$}
\put(104,32){\Large $=$}
\put(188,32){\Large $=$}
\put(248,32){\large $\in \Pi_2(X).$}

\put(-125,32){\Large $\CC $}

\put(-58,32){\Large $\CC $}
\put(-50,54){\large $h $}
\put(-21,32){\large $g_1 $}

\put(49,32){\Large $\CC $}
\put(79,54){\large $h $}
\put(85,20){\large $g_1 $}

\put(138,32){\Large $\CC ' $}
\put(130,56){\large $h \lhd g_1 $}
\put(178,20){\large $g_1 $}

\put(220,32){\Large $\CC ' $}
\put(212,58){\large $h \lhd g_1 $}

\end{picture}
\end{center}
\vskip -1.9937pc
%\caption{The coloring conditions at each crossing and around arcs. \label{fig.color}}
\end{figure}

\noindent 
Here the first and forth equalities are obtained from the concordance relation, 
and in the second (resp. third) equality the loop colored by $g_1$ passes under (resp. over) 
the all arc of $D$. Note that we here only use Reidemeister moves. 
Hence, iterating this process, we have a based $\X$-coloring $\CC^{(n)}$ of the $D$ such that 
$[\CC ]=[\CC ^{(n)}] \in \Pi_2(\X)$. % and the arc $\gamma_0$ is colored by $1_{\X}$. 
\end{proof}

\begin{lem}\label{bap}
Let $X$ be a connected quandle of type $t_X $.
Let $\iota: \Ker (\epsilon_X) \ra \As (X)$ be the inclusion \eqref{eseq}, 
and $\pi_{\rho}: C_3^{\gr } (\Ker (\epsilon_X )) \ra C_3^{\gr } (\Ker (\epsilon_X ))_{\Z} $ be the projection. 
Then there is an isomorphism $ \xi:H_3^{\gr } (\As (X)) \ra H_3^{\gr } (\Ker (\epsilon_X ))_{\Z}$ modulo $t_X $-torsion such 
that $ \xi \circ \iota_*=t_X \cdot (\pi_{\rho})_*$. 
\end{lem}
\begin{proof} 
Fix $x \in X,$ and consider the subgroup $ \langle e_{x}^{ n t_X }\rangle_{ n \in \Z} $ of $ \As (X)$, 
which is contained in the center (see \cite[Lemma 4.1]{Nos2}). Put the quotient $Q_X:= \As (X) / \langle e_{x}^{ n t_X }\rangle_{ n \in \Z} $. 
By the Lyndon-Hochshild spectral sequence, the projection induces an isomorphism $P_*: H_3^{\gr }(\As(X)) \cong H_3^{\gr }(Q_X )$ up to $t_X$-torsion,
since the $H_2^{\gr}(\As (X))$ is shown to be annihilated by $t_X$ \cite[Corollary 6.4]{Nos2}.
Furthermore, noting the group extension $ \Ker (\epsilon_X) \ra G_X \ra \Z/ t_X $, 
the transfer gives an isomorphism $\mathcal{T}: H_3^{\gr }(Q_X ) \ra H_3^{\gr }( \Ker (\epsilon_X))_{\Z } $ modulo $t_X$; see \cite[\S III.10]{Bro}. 
Hence, denoting $ \mathcal{T} \circ P_*$ by $\xi $, we have the equality $ \xi \circ \iota_*= t_X \cdot (\pi_{\rho})_*$ by construction. 
%by the Lyndon-Hochshild spectral sequence, we have $ H_3^{\mathrm{gr}}(\As(X)) \cong H_3^{\mathrm{gr}}( G_Q)$. 
%Furthermore, considering $\epsilon_X: G_Q \ra \Z/m$, the transfer map gives an isomorphism $ H_3^{\mathrm{gr}}( G_Q) \cong H_3^{\mathrm{gr}}(\Ker(\epsilon_X) )_{\Z} $.
\end{proof}

\section{ Proofs of Proposition \ref{prop2} and Theorem \ref{thm2} }\label{reviewsw}
This section proves Proposition \ref{prop2} and Theorem \ref{thm2}. % about the chain map $\Phi_3$ from quandle cohomologies. 
The outline of the proof is as follows. 
With respect to an Alexander quandle over $\F$, 
a basis of the third cohomology over $\F$ was found by Mochizuki \cite{Moc2}, which we review in \S \ref{reMoc}.
Furthermore, we will see that it is enough for the isomorphisms stated in Theorem \ref{thm2} to show these surjectivity. 
So, we will construct group 3-cocycles of $\As(X)$ as 
preimages of the basis via the chain maps $\Phi_3 $ and $\widetilde{\Phi}_3 $ (see \S \ref{1305}).

% we fix $\omega \in \F$ with $\ww \neq 0,1$,
% and assume that $X$ is the connected Alexander quandle on $\F$. 

To accomplish this outline, we start reviewing a simple presentation of $\As (X)$ of a connected Alexander quandle $X$, shown by Clauwens \cite{Cla}.
Set a homomorphism $\mu_X : X \otimes X \ra X \otimes X$ defined by 
\[ \mu_X( x \otimes y)= x\otimes y - T y \otimes x. \]
Using this $ \mu_X$, let us equip $\Z \times X \times \mathrm{Coker}( \mu_X)$ with a group operation given by 
\begin{equation}\label{clauwens} (n, a , \kappa) \cdot (m, b, \nu )= (n+m, \ T^m a +b, \ \kappa + \nu + [T^m a \otimes b]). \ \end{equation}
Then a homomorphism 
$\As(X) \ra \Z \times X \times \mathrm{Coker}( \mu_X) $ sending the generators $e_x $ to $(1,x,0)$ is isomorphic \cite[Theorem 1]{Cla}.
The lower central series of $\As(X)$ are then described as 
\begin{equation}\label{lower} \As(X) \supset X \times \mathrm{Coker}( \mu_X) \supset \mathrm{Coker}( \mu_X) \supset 0 . \end{equation}
%Thanks to his presentation of $\As(X)$, we easily show
In particular, the kernel $\Ker (\epsilon_X) $ in \eqref{eseq} is 
the subgroup on the set $ X \times \mathrm{Coker}( \mu_X) $. 
Incidentally, an isomorphism $H_2^Q(X) \cong \mathrm{Coker}( \mu_X)$ is shown \cite{Cla}. 
%The reader should keep in mind the presentation of $\As(X)$, and that the subgroup on $ $ coincides with $G_X$ in \S \ref{a32das06}.

\vskip 0.7pc
\noindent
{\bf Notation} Denote by $G_X$ the subgroup $\Ker (\epsilon_X)$ on $X \times \mathrm{Coker}( \mu_X) $. 
From now on, in this section, we let $X$ be an Alexander quandle on $\F$ with $\omega \in \F$.
Let $X$ be of type $t_X$. 
That is, $t_X$ is the minimal satisfying $\omega^{t_X}=1$. Note that $t_X $ is coprime to $q$ since $\omega^{q-1}=1$.

\subsection{Review of Mochizuki 3-cocycles}\label{reMoc}
We will review Mochizuki 2-, 3-cocycles of $X = \F $. 
We here regard polynomials in the ring $\F[U_1, \dots, U_n]$ as functions from $X^n$ to $\F$, and as being in the complex $C_{n}^{Q_G}(X;\F)$ in \S \ref{SScon31}.

\begin{thm}[{\cite[Lemma 3.7]{Moc2}}]\label{mochizukiteir3i} 
The following set represents a basis of $H_Q^2(X ;\F)$. 
\[ \{ \ U_1^{q_1} U_2^{q_2} \ | \ \ww^{q_1+q_2}=1, \ 1 \leq q_1<q_2<q ,\ \mathrm{and } \ q_i \ \mathrm{ \ is \ a \ power \ of \ }p.\ \}.\]\end{thm}

Next, we describe all the quandle 3-cocycles of $X$.
To see this, recall the following three polynomials over $\F$ (\cite[\S 2.2]{Moc2}):
\[ \chi (U_j,U_{j+1}):= \!\! \sum_{1 \leq i \leq p-1} \!\! (-1)^{i-1} i^{-1} U_{j}^{p-i} U_{j+1}^i= \bigl( (U_{i-1} +U_i)^p - U_{i-1}^p -U_{i}^p \bigr) /p, \]
\[ E_0(a \cdot p,b) := \bigl( \chi ( \omega U_1 ,U_2)- \chi (U_1 ,U_2)\bigr)^a \cdot U_3^b, \ \ \ E_1(a,b \cdot p) := U_1^a \cdot \bigl( \chi ( U_2 ,U_3)- \chi ( \omega^{-1} \cdot U_2 ,U_3)\bigr)^b. \]
Define the following set $I^{+}_{q,\ww}$ consisting of the polynomials under 
some conditions:
\begin{eqnarray}\label{sgsgsg}
& I^{+}_{q,\ww}:=& \{ E_0 (q_1 \cdot p,q_2 ) \ | \ \ww^{ p \cdot q_1 + q_2 }=1, \ q_1 <q_2 \} \ \cup \ \{ E_1(q_1 ,q_2 \cdot p) \ | \ \ww^{q_1 + p\cdot q_2}=1, \ q_1 \leq q_2 \} \nonumber \\
& & \cup \ \{ U_1^{q_1} U_2^{q_2} U_3^{q_3} \ | \ \ww^{q_1 + q_2+q_3}=1, \ q_1 <q_2<q_3 \}.
\end{eqnarray}
Here the symbols $q_i$ range over powers of $p$ with $q_i <q$. 

Furthermore, we review polynomials denoted by $ \Gamma(q_1,q_2,q_3,q_4)$. 
For this, we define a set 
$ \mathcal{Q}_{q,\ww} \subset \Z^4$ consisting of quadruples $(q_1,q_2,q_3,q_4)$ such that 

\begin{itemize}
\item $q_2 \leq q_3, \ q_1 <q_3, \ q_2 <q_4, $ and $ \ww^{q_1 +q_3}= \ww^{q_2 +q_4}=1.$ Here, if $p=2$, we omit $q_2 = q_3$.
\item One of the following holds: {\bf Case 1} \ $\ww^{q_1+q_2}=1.$

\ \ \ \ \ \ \ \ \ \ \ \ \ \ \ \ \ \ \ \ \ \ \ \ \ \ \ \ \ \ \ \ \ \ \ \ \ {\bf Case 2} \ $\ww^{q_1+q_2}\neq 1$, and $q_3 >q_4$.

\ \ \ \ \ \ \ \ \ \ \ \ \ \ \ \ \ \ \ \ \ \ \ \ \ \ \ \ \ \ \ \ \ \ \ \ \ {\bf Case 3} \ $(p\neq 2)$, $\ww^{q_1+q_2}\neq1$, and $q_3=q_4$.

\ \ \ \ \ \ \ \ \ \ \ \ \ \ \ \ \ \ \ \ \ \ \ \ \ \ \ \ \ \ \ \ \ \ \ \ \ {\bf Case 4} \ $(p\neq 2)$, $\ww^{q_1+q_2}\neq1$, $q_2 \leq q_1 <q_3 <q_4, \ww^{q_1}=\ww^{q_2}$.

\ \ \ \ \ \ \ \ \ \ \ \ \ \ \ \ \ \ \ \ \ \ \ \ \ \ \ \ \ \ \ \ \ \ \ \ \ {\bf Case 5} \ $(p=2)$, $\ww^{q_1+q_2}\neq1$, $q_2 < q_1 <q_3 <q_4 , \ww^{q_1}=\ww^{q_2}$.

\end{itemize}

\noindent
For $(q_1,q_2,q_3,q_4) \in \mathcal{Q}_{q,\ww}$ in each cases, the polynomial $ \Gamma(q_1,q_2,q_3,q_4)$ is defined as follows\footnote{In Cases 3, 4 and 5, we change the forms of $\Gamma(q_1,q_2,q_3,q_4)$ in \cite{Moc2}; however, our $\Gamma$ are cohomlogous to the original ones.}:
\[ {\bf Case \ 1} \ \ \ \ \Gamma(q_1,q_2,q_3,q_4):= U_1^{q_1} U_2^{ q_2+q_3} U_3^{q_4}. \]
\[ {\bf Case \ 2} \ \ \ \ \Gamma(q_1,q_2,q_3,q_4):= U_1^{q_1} U_2^{ q_2+q_3} U_3^{q_4}- U_1^{q_2} U_2^{ q_1+q_4} U_3^{q_3} \]
\[ \ \ \ \ \ \ \ \ \ \ \ \ \ \ \ \ \ \ \ \ \ \ \ \ \ \ \ \ \ \ \ \ \ \ \ \ \ \ - (\ww^{q_2}-1)^{-1}(1-\ww^{q_1+q_2}) ( U_1^{q_1} U_2^{ q_2} U_3^{q_3 + q_4}- U_1^{q_1 +q_2 } U_2^{ q_4} U_3^{q_3}) .\]
\[ {\bf Case \ 3} \ \ \ \ \Gamma(q_1,q_2,q_3,q_4):=U_1^{q_1} U_2^{ q_3+q_4} U_3^{q_2} .\]
\[{\bf Case \ 4} \ \mathrm{and} {\bf \ Case \ 5} \ \ \ \ \Gamma(q_1,q_2,q_3,q_4):=U_1^{q_3} U_2^{ q_1+q_2} U_3^{q_4}.\]

\begin{rem}\label{chuui} 
The 3-cocycle in Case 3 (resp. 4 and 5) is obtained from that in Case 1 
after changing the indices $ (1,2,3,4)$ to $(1,3,4,2)$ $\bigl($resp. to $(3,1,2,4 )$ $\bigr)$.
\end{rem} 
We call the set $ \mathcal{Q}_{q,\ww}$ {\it Mochizuki quadruples}. Then we state the main theorem in \cite{Moc2}: 
% Then 

\begin{thm}[{\cite{Moc2}}]\label{mochizukiteiri} 
The third cohomology $H_Q^3(X ;\F)$ is spanned by the following set composed of non-trivial 3-cocycles.
Here $q_i $ means a power of $p$ with $q_i <q$.
\[ I^+_{q,\ww} \ \cup \{ \Gamma(q_1,q_2,q_3,q_4) \ | \ (q_1,q_2,q_3,q_4) \in \mathcal{Q}_{q,\ww}\} \cup \{ \ U_1^{q_1} U_2^{q_2} \ | \ \ww^{q_1 + q_2}=1, \ q_1 <q_2 \} .\]
\end{thm}
\begin{rem}\label{chuu456i} 
%Unfortunely, the original paper (see \cite[Remark 6.2]{Nosaka3}),
Unfortunately 
the original statement and his proof of this theorem contained slight errors, which had however been corrected by Mandemaker \cite{Man}.
\end{rem}
\subsection{ Proof of Theorem \ref{thm2} }\label{1305}%Furthermore, the lifted map $\widetilde{\varphi}_3 : C_n^R(\X ) \ra C_n^{\rm gr}(G_X )_{\Z} $ is formulated as 
%\begin{equation}\label{abdl267} \notag
%\widetilde{\varphi}_3 (a, \alpha, b , \beta, c, \gamma ) = (a, \alpha, b , \beta, c, \gamma ) - (\ww a, \alpha, b , \beta, c, \gamma ) - (\ww a, \alpha, \ww b , \beta, c, \gamma ) + (\ww ^2a, \alpha, \ww b , \beta, c, \gamma ) .
% \end{equation}
First, to prove Proposition \ref{prop2}, we prepare a lemma for a study of the quandle 3-cocycles in \eqref{sgsgsg}.
\begin{lem}\label{group3cohomology}
Let us identify $G=(\Z_p)^h$ with $\F$ as an additive group. 
Then the second group cohomology $H_{\gr}^2(G;\F) \cong (\F)^{\frac{h(h+1)}{2}}$ is 
generated by the following group 2-cocycles: 
$$ \{ \ U_1^{q_1} U_2^{q_2} \ | \ 1 \leq q_1<q_2<q ,\ \mathrm{where} \ q_i \ \mathrm{ is \ a \ power \ of \ }p.\ \}. $$
\noindent
Furthermore, the third one $H_{\gr}^3(G;\F) \cong \F^{\frac{h(h+1)(h+2)}{6}}$ is 
spanned by the following 3-cocycles: 
\begin{eqnarray}
& \{ \ U_1^{q_1} U_2^{q_2} U_3^{q_3} \ | \ q_1 < q_2 <q_3 \ \} \ \cup \ \{ ((U_1 + U_2)^{q_1}-U_1^{q_1}-U_2^{q_1} )\cdot U_3^{q_2}/p \ | \ q_1 < q_2 \ \} 
\nonumber \\
& \cup \ \{ \ U_1^{q_1}( (U_2 + U_3)^{q_2} - U_2^{q_2} -U_3^{q_2})/p \ | \ q_1 \leq q_2\ \} \nonumber ,
\end{eqnarray}
where $q_1,q_2,q_3$ run over powers of $p$ with $1 \leq q_j <q.$
Moreover, regarding the multiplication of $\ww \in \F$ as an action of $\Z$ on $\F$, 
the $\Z$-invariant parts $H^i_{\gr}(G;\F)^{\Z}$ are generated by the above polynomials of degree $d$ satisfying $\ww^d=1. $ Here $i=2,3. $
\end{lem} 
\noindent 
The group cohomologies of abelian groups can be calculated in many ways, e.g., 
by similar calculations to \cite[V.\S 6]{Bro} or \cite{Moc2};
So we omit proving Lemma \ref{group3cohomology}.

Returning to our subject, we apply these generators in Lemma \ref{group3cohomology} to the pullback of the chain map $\varphi_3$ (see Definition \ref{defIKlift}). 
Then easy computations show the identities
\[ \varphi_3^*( U_1^{q_1} U_2^{q_2} U_3^{q_3} ) = t_X (1- \ww^{q_1})(1- \ww^{q_1 + q_2}) \cdot F ( q_1,q_2,q_3), \]
\[ \varphi_3^*\bigl( (U_1 + U_2)^{q_1} - U_1^{q_1} -U_2^{q_1})U_3^{q_2}/p \bigr) = t_X (1- \ww^{q_2}) \cdot E_0(q_1,q_2), \]
\begin{equation}\label{jijii} \varphi_3^* \bigl( U_1^{q_1}( (U_2 + U_3)^{q_2} - U_2^{q_2} -U_3^{q_2})/p \bigr) = t_X (1- \ww^{q_1}) \cdot E_1 ( q_1,q_2) \in C^3_Q(X;\F). \end{equation}
Compared with the way in \cite{Moc2} that the right quandle 3-cocycles were found as solutions of a differential equation over $\F$, 
the three identities via the map $\varphi_3^* $ are simple and miraculous. 

Using the identities we will prove Proposition \ref{prop2} as follows:
\begin{proof}[Proof of Proposition \ref{prop2}.]
The injectivity of $\Phi_3^* =( \pi_{\rho} \circ \varphi_3)^*$ follows from that 
this $\Phi_3^*$ gives a 1:1 correspondence between a basis of $H^3_{\gr}((\Z_p)^h ;\F)^{\Z}$ and 
a basis of a subspace of $H^3_Q(X;\F)$ because of the previous three identities (compare Theorem \ref{mochizukiteiri} with Lemma \ref{group3cohomology}). 

Next, assume $H^2_Q(X;\F)=0$. Then, Theorem 
\ref{mochizukiteir3i} implies that no pair $(q_1,q_2)$ satisfies $ \ww^{q_1 +q_2}=1$ and $q_1 < q_2 <q$. 
Hence, by observing Theorem \ref{mochizukiteiri} carefully, the $H^3_Q(X;\F) $ is generated by 
the image of $\Phi_3^* $. Therefore $\Phi_3^* $ is an isomorphism as desired.
\end{proof}

Next, we will prove Theorem \ref{thm2}. To this end, we now observe the cokernel $ \Coker (\Phi_3^* )$. % for more details.

To begin, we study the chain map $(\Phi_2 \circ \mathcal{P} )^* : H^2_{\gr}((\Z_p)^h ;\F)^{\Z} \ra H^3_Q(X)$ stated in Proposition \ref{aip}. 
Recall from Lemma \ref{group3cohomology} that this domain is generated by polynomials of the form $U_1 ^{q_1} U_2^{q_2}$.
So, recalling the composite $ \Phi_2 \circ \mathcal{P} $ from Proposition \ref{aip}, we easily see 
$$ (\Phi_2 \circ \mathcal{P} )^* (U_1 ^{q_1} U_2^{q_2}) = t_X (1 - \ww^{q_1})U_1 ^{q_1} U_2^{q_2} \in C^3_{Q_G}(X;\F). $$
Hence, the third term in Theorem \ref{mochizukiteiri} is spanned by the image of this map $ (\Phi_2 \circ \mathcal{P} )^* $.

Furthermore, we will discuss the cokernel of $\Phi_3^* \oplus ( \Phi_2 \circ \mathcal{P} )^* $.
By observing Theorem \ref{mochizukiteiri} carefully, we see that a basis of the cokernel consists of 
the polynomials $\Gamma$'s coming from the Mochizuki quadruples $ \mathcal{Q}_{q,\ww}$. 
Let us denote a quadruple $(q_1,q_2,q_3,q_4) \in \mathcal{Q}_{q,\ww}$ by $\mathfrak{q}$ for short. 
Case by case, we now introduce a map $\theta_{\Gamma}^{\mathfrak{q}}: (G_X)^3 \ra \F$ by setting the values of $\theta_{\Gamma}^{ \mathfrak{q}}$ at $(x, a \otimes b,y, c \otimes d ,z, e \otimes f ) \in (X \times \Coker (\mu_X))^3$ 
as follows. In Case $1$, $\theta_{\Gamma}^{ \mathfrak{q}}$ is defined by the formula
\[ (1-\ww)^{-q_2} \bigl( x^{q_1 } y^{q_2+q_3} + x^{q_1 +q_3} y^{q_2} - (1-\ww)^{-q_2} ( \ww^{q_2} a^{q_1} b^{q_2} + a^{q_2} b^{q_1} - x^{q_1+q_2} ) y^{q_3} \]
\begin{equation}\label{sana} \ \ \ \ \ \ \ \ \ \ \ \ \ \ \ + (1 - \ww)^{-q_1} (a^{q_1}b^{q_3} + \ww^{q_1} a^{q_3}b^{q_1}- x^{q_1+q_3})y^{q_2} \bigr) z^{q_4}. \end{equation}
In Case $2$, the value of $\theta_{\Gamma}^{ \mathfrak{q}}$ is given by the formula
\[(1 - \ww)^{-q_1-q_2}\bigl( x^{q_1} ( y^{q_2+q_3}z^{q_4} +y^{q_2}z^{q_3 + q_4} ) - ( x^{q_1 +q_2 }y^{q_4} + x^{q_2} y^{q_1+q_4} ) z^{q_3} \]
\[ \ \ \ \ \ \ +(1-\ww)^{- q_3 } ( x^{q_1+q_3} - \ww^{q_3} a^{q_1}b^{q_3}- a^{q_3} b^{q_1})y^{q_2} z^{q_4} -(1-\ww)^{-q_4} ( x^{q_2+q_4} - \ww^{q_4} a^{q_2}b^{q_4}- a^{q_4} b^{q_2})y^{q_1} z^{q_3} \bigr). \]
Furthermore, for Case $3$ (resp. $4$ and $5$), the value
is defined to be that of Case $1$ by changing the indices $ (1,2,3,4)$ to $(1,3,4,2)$ $\bigl($resp. to ($3,1,2,4 $) $\bigr)$, according to Remark \ref{chuui}.

\begin{lem}\label{bababa1} For $ \mathfrak{q}=(q_1,q_2,q_3,q_4) \in \mathcal{Q}_{q,\ww}$, the map $\theta_{\Gamma}^{\mathfrak{q}}$ from $(G_X)^3 $ to $\F$ is 
a $\Z$-invariant group $3 $-cocycle of $G_X$.
%, that is, this map $\theta_{\Gamma}^{\mathfrak{q}}$ satisfies 
%the equalities \eqref{abdasd7} and \eqref{abdasd8}. 

Moreover, using the map $\widetilde{\Phi}_3$, the pullback 
$\widetilde{\Phi}_3^*( \theta_{\Gamma}^{\mathfrak{q}})$ equals $t_X \cdot p_X^*( \Gamma(\mathfrak{q} ))$ in $C^3_Q(\X ;\F)$.
\end{lem}
\begin{proof} %For the proof, we first study group 3-cocycles of the group $G_X$. %For $n=3$, we concretely rewrite the dual of the lifted map $\widetilde{\varphi}_3$. 
Note that a map $\theta : (G_X)^3 \ra A$ is a $\Z$-invariant group 3-cocycle, by definition, if and only if it satisfies the two equalities
\begin{equation}\label{abdasd7} \theta (\mathfrak{b},\mathfrak{c},\mathfrak{d})- \theta (\mathfrak{a}\mathfrak{b},\mathfrak{c},\mathfrak{d})+ \theta (\mathfrak{a}, \mathfrak{b}\mathfrak{c},\mathfrak{d})- \theta (\mathfrak{a},\mathfrak{b}, \mathfrak{c}\mathfrak{d}) + \theta (\mathfrak{a}, \mathfrak{b},\mathfrak{c})=0, \notag \end{equation}
\begin{equation}\label{abdasd8} \theta ((\ww a,\alpha) ,(\ww b ,\beta),(\ww c,\gamma) )= \theta (( a,\alpha) ,(b ,\beta),(c,\gamma) ), \notag \end{equation}
for any $\mathfrak{a}=(a,\alpha) ,\ \mathfrak{b}=(b ,\beta), \ \mathfrak{c}=(c,\gamma), \ \mathfrak{d} =(d,\delta) \in G_X =X \times \Coker (\mu_X)$.
Then, by elementary and direct computations, it can be seen that the maps $\theta_{\Gamma}^{\mathfrak{q}}$ are $\Z$-invariant group $3 $-cocycles of $G_X$.
Also, the desired equality $ \widetilde{\Phi}_3^*( \theta_{\Gamma}^{\mathfrak{q}}) = t_X \cdot p_X^*( \Gamma(\mathfrak{q} )) $ can be obtained by a direct calculation.
\end{proof}

\begin{proof}[Proof of Theorem \ref{thm2}.]
Let $q$ be odd. 
As is known \cite[Lemma 9.15]{Nos2}, the induced map $p_X^*: H^3_Q(X ;\F) \ra H^3_Q(\X ;\F) $ is surjective. 
Hence, according to Lemma \ref{bababa1}, there exists a section $\mathfrak{s} : H^3_Q(\X ;\F) \ra H^3_Q(X ;\F) $ such 
that $\mathfrak{s} (\widetilde{\Phi}_3^*( \theta_{\Gamma}^{\mathfrak{q}}))=\Gamma(\mathfrak{q} ) $ for any $\mathfrak{q} \in \mathcal{Q}_{q,\ww}$. 
To summarize the above discussion, the sum 
$( (\Phi_2 \circ \mathcal{P} )^* \oplus \Phi_3^* ) \oplus \bigl( \mathfrak{s} \circ \mathrm{res}(\widetilde{\Phi}^*_3 )\bigr)$ in \eqref{ac} 
is an isomorphism to $ H^3_Q(X ;\F) $. 
\end{proof}

%\begin{lem}\label{bababa131} Let $X$ be a connected Alexander quandle of odd order.
%For any prime $\ell$ which is coprime to the type $t_X$, then the the pushfoward $(p_X)_*: H^Q_i(\X ;\Z) \ra H^Q_i(X ;\Z) $ is an injection after $\ell$-localization, where $i=2,3$. 
%\end{lem}
%\begin{proof} The injectivity with $i=2$ is obvious, since $ H^Q_2(\X ;\Z)$ is known to be annihilated by $t_X$ (see \cite[Theorem \ref{}]{Nos2}). 
%Next, when the case $i=3$. 
%\end{proof}

\

Incidentally, we will show that the group 3-cocycles $ \theta_{\Gamma}^{ \mathfrak{q}}$ above except Case 2 are presented by Massey products. 
To see this, we consider a group homomorphism 
$$f^{q_i}: G_X \ra \F; \ \ \ \ (x, \alpha) \longmapsto x^{q_i},$$ 
which is a group 1-cocycle of $G_X$.
For group 1-cocycles $f,g$ and $h$, we denote by $f \wedge g$ the cup product; further, if $ f \wedge g=g\wedge h =0 \in H^2_{\gr}(G_X;\F)$,
we denote by $ <f,g,h>$ the triple Massey product in $H^2_{\gr}(G_X;\F )$ as usual (see, e.g., \cite{Kra} for the definition).
\begin{prop}\label{lll}
Let $e \neq 2$. Let $(q_1,q_2,q_3,q_4) \in \mathcal{Q}_{q,\ww}$ satisfy {\bf Case $e$} in \S \ref{reMoc}. 
The group 3-cocycle $ \theta_{\Gamma}$ described above is of the followings form in the cohomology $ H^3_{\rm gr}( G_X ;\F)$. 
\begin{eqnarray*}\displaystyle{ H^3_{\rm gr}( G_X ;\F) \ni \theta_{\Gamma}} = \left\{ \begin{array}{lll}
\displaystyle{ (1-\ww^{q_2})^{-1}< f^{q_3}, f^{q_1},f^{q_2}>\!\! \wedge f^{q_4} } & & \ \ \mathrm{for} \ \ e=1, \\
\displaystyle{ (1-\ww^{q_3})^{-1}< f^{q_4}, f^{q_1},f^{q_3}>\!\! \wedge f^{q_2} } & & \ \ \mathrm{for} \ \ e=3, \\
\displaystyle{ (1-\ww^{q_3})^{-1}< f^{q_1}, f^{q_2},f^{q_3}>\!\! \wedge f^{q_4} } & & \ \ \mathrm{for} \ \ e=4 \mathrm{\ \ or \ \ }5. \\
\end{array} \right. \end{eqnarray*} 
\end{prop}
\begin{proof}We use notation $(x, a \otimes b,y, c \otimes d ,z, e \otimes f ) \in (X \times \Coker (\mu_X))^3$ as above.
Notice first that the cup product $f^{q_1} \wedge f^{q_2}$ is the usual product $x^{q_1}y^{q_2} $ (see \cite[V.\S 3]{Bro}).
For {\bf Case 1}, we now calculate the Massey product $<f^{q_3},f^{q_1},f^{q_2} >$. We easily check two equalities
\[ x^{q_3} y^{q_1}= (1-\ww)^{-q_1} \delta_1 \bigl( a^{q_1} b^{q_3} + \ww^{q_1} a^{q_3} b^{q_1} - x^{q_1 +q_3} \bigr), \]
\[ x^{q_1}y^{q_2}= (1-\ww)^{-q_2 } \delta_1 \bigl( \ww^{q_2} a^{q_1} b^{q_2} + a^{q_2} b^{q_1} - x^{q_1 +q_2 } \bigr). \]
Hence, from the definition of Massey products, $<f^{q_3},f^{q_1},f^{q_2} >$ is represented by 
$$ (1-\ww)^{-q_1} ( a^{q_1} b^{q_3} + \ww^{q_1} a^{q_3} b^{q_1}- x^{q_1 +q_3} ) y^{q_2}+ (1- \ww)^{-q_2} x^{q_3} ( \ww^{q_2} c^{q_1} d^{q_2} + c^{q_2} d^{q_1}-y^{q_1 +q_2} ). $$
Furthermore, we set a group 2-cocycle defined by 
\[ \mathcal{F}:= (1-\ww)^{-q_2} \bigl( <f^{q_3},f^{q_1},f^{q_2} > + (1-\ww)^{-q_2} \delta_1 ( \ww^{q_2} x^{q_3} a^{q_1} b^{q_2} + x^{q_3}a^{q_2} b^{q_1} - x^{q_1+q_2+q_3} ) \bigr) . \]
A direct calculation then shows the equality $ \mathcal{F} \cdot z^{q_4} = \theta_{\Gamma}^{\qq}$ by definitions, i.e., $ <f^{q_3},f^{q_1},f^{q_2} > \wedge f^{q_4}= \theta_{\Gamma}^{\qq} \in H^3_{\gr}(G_X ; \F)$ as desired.

Similarly, the same calculation holds for {\bf Cases 3, 4, 5} according to Remark \ref{chuui}.
\end{proof}
However, a geometric meaning of the cocycle $\theta_{\Gamma}^{\qq} $ with Case $2$ remains to be open.

\section{Some calculations of shadow cocycle invariants}\label{k1ti2}

As an application of Theorem \ref{thm3}, we will compute some $\Z$-equivariant parts of Dijkgraaf-Witten invariants, which is equivalent to a shadow cocycle invariant.
In this section, we confirm ourselves to Alexander quandles on $\F$ with $\omega \in \F $.
Recall from Lemma \ref{bababa1} that the quandle 3-cocycles $\Gamma(q_1,q_2,q_3,q_4)$ found by Mochizuki (see \S \ref{reMoc} the definition)
are derived not from group cohomologies of abelian groups, but from that of the non-abelian group $G_X$.
So we focus on the cocycles, and fix some notation. 
Let $\mathfrak{q}$ denote a Mochizuki quadruple $(q_1,q_2,q_3,q_4) $ in $\mathcal{Q}_{q,\ww}$ for short, and replace $\Gamma(q_1,q_2,q_3,q_4)$ by $\Gamma(\mathfrak{q} )_e$, if $\mathfrak{q}$ satisfies Case $e$ in \S \ref{reMoc} ($e \leq 5$).

Incidentally, the set of $X$-colorings was well-studied. In fact, if $D$ is a diagram of a knot $K$, then there is a bijection
\begin{equation}\label{Inoue} \col_X(D) \longleftrightarrow X \oplus \bigoplus_{i=1} \F[T]/ (T-\ww, \Delta_i(T)/ \Delta_{i+1}(T)),\end{equation} 
where $\Delta_i(T)$ is the $i$-th Alexander polynomial of $K$ (see \cite{Ino}). 
Therefore, we shall study weights in the cocycles invariants.

\subsection{The cocycle invariants of torus knots constructed from $\Gamma(q_1,q_2,q_3,q_4)$ }\label{k41ti2}

This subsection considers the torus knots $T(m,n)$.
We here remark that $m$ and $ n $ are relatively prime and the isotopy $ T(m,n) \simeq T(n,m ) $; 
thereby $n$ may be relatively prime to $p$ without loss of generality.
We determine all of the values of the invariants for $T(m,n)$ as follows:

\begin{thm}\label{torus}
%Let $m, n \in \Z$. 
Let $q$ be relatively prime to $n$. 
Let $T(m,n)$ be the torus knot.
Let $\mathfrak{q} \in \mathcal{Q}_{q, \omega}$ be a Mochizuki quadruple, and $\Gamma (\mathfrak{q} )_e$ be the associated quandle 3-cocycle. 
Then the quandle cocycle invariant $I_{\Gamma (\qq)_e }(T(m,n))$ is expressed by one of the following formulas: 
\begin{enumerate}[(i)]
\item If $e=1$, $\ww^{mn}=1$, $\ww^m \neq 1$ and $\ww^{n} \neq 1$, then
\begin{equation}\label{wwneq1} 
I_{ \Gamma (\mathfrak{q})_1 }\bigl( T(m,n) \bigr)= q^2 \sum_{a \in \F }1_{\Z} \{ -2 mn \frac{ (\zeta -\ww )^{q_2+q_3} \ww^{q_4}
}{(1-\zeta)^{q_2+q_3} }\cdot a^{q_1 +q_2+q_3 +q_4} 
\} \in \Z[\F], \end{equation}
where $\zeta$ is the $n$-th primitive root of unity satisfying $\ww^m=\zeta^m$.
Furthermore, if $e=3$ (resp. $4$ or $5$), then the value of $ I_{ \Gamma (\mathfrak{q})_e }$ is obtained from the above value $I_{ \Gamma (\mathfrak{q})_1 } $
after changing the indices $ (1,2,3,4)$ to $(1,3,4,2)$ $\bigl($resp. to $(3,1,2,4 )$ $\bigr)$ such as Remark \ref{chuui}.
\item Let $p=2$ or $3$, and let $e=1$. If $ \ww^n=1$ and if $m$ is divisible by $p$, then
\begin{equation}\label{xwwedqual1}
I_{\Gamma (\qq)_1 }(T(m,n))= q^2 \displaystyle{ \sum_{a \in \F} 1_{\Z }\bigr\{ \frac{mn}{p} (1- \ww )^{q_3 +q_4} a^{q_1 +q_2 +q_3 +q_4} \bigr\} \in \Z[\F]}.\\
\end{equation}
Furthermore, if $e=3$ (resp. $4$ or $5$), then the value $ I_{ \Gamma (\mathfrak{q})_e }$ is obtained from the value $I_{ \Gamma (\mathfrak{q})_1 } $
after changing the indices $ (1,2,3,4)$ to $(1,3,4,2)$ $\bigl($resp. to $(3,1,2,4 )$ $\bigr)$, similarly.

\item Let $e=2$. If $p=2$, $ \ww^n=1$ and if $m$ is divisible by $2$, then $I_{\Gamma (\qq)_2 }(T(m,n))$ is equal to 
$ q \sum_{a , \delta \in \F} 1_{\Z } \{ mn \mathcal{E}_2(a,\delta)/2\} \in \Z[\F ]$. Here $\mathcal{E}_2(a,\delta) \in \F$ is temporarily defined by
\begin{equation}\label{xwwedqual15}
a^{q_2+q_3} \bigl( (1+\ww^{q_1}) a^{q_1} \delta^{q_4} + (1+\ww^{q_4}) a^{q_4 } \delta^{q_1} \bigr) + a^{q_1+q_4} \bigl( (1+\ww^{q_2}) a^{q_2} \delta^{q_3} + (1+\ww^{q_3}) a^{q_3 } \delta^{q_2} \bigr) . \notag
\end{equation}

\item Otherwise, the invariant is trivial. Namely, $I_{\Gamma (\qq)_e }(T(m,n)) \in \Z$. \end{enumerate}
\end{thm}
\noindent
This is proved in \S \ref{proprof}. Note that, for $e=2$, the invariant is non-trivial in only the case (iii).

\begin{rem}\label{torusrem2365}
Asami and Kuga \cite[\S 5.2]{AK} partially calculated some values of $I_{\Gamma (\qq)_e }(T(m,n))$ in the only case $\F = \mathbb{F}_{5^2}$ and $n=3$, by the help of computer.
\end{rem}

We consider the $t $-fold cyclic covering of $S^3$ branched over $T(m,n)$.
This is the Brieskorn manifold $\Sigma(m,n, t)$; see \cite{Milnor}.
Hence, we obtain a $\Z$-equivariant part of the Dijkgraaf-Witten invariant of $\Sigma(m,n,t )$.

\begin{cor}\label{torus314}
Let $m, n$ be coprime integers. 
Let $X$ be of type $t$.
Let a Mochizuki quadruple $(q_1, q_2, q_3, q_4) \in \mathcal{Q}_{q,\ww} $ satisfy Case 1, and $\theta_{\Gamma} \in H_{\gr}^3(G_X;\F)$ be the group 3-cocycle in Lemma \ref{bababa1}. 
Let $p>2$ be coprime to $n$ and to $t$. 
If $\ww^{mn}=1$, $\ww^{n} \neq 1$ and $\ww^{m}\neq 1$, then
\[ \mathrm{DW}^{\Z}_{\theta_{\Gamma}} \bigl( \Sigma(m,n, t ) \bigr)= \sum_{a \in \F} 1_{\Z } \{ -2 t mn \frac{ (\zeta -\ww )^{q_2+q_3} \ww^{q_4}
}{(1-\zeta)^{q_2+q_3} }
a^{q_1 +q_2+q_3 +q_4} \} \ \in \Z[\F]. \]
\end{cor}
\noindent
Proposition \ref{lll} says that the cocycle $ \theta_{\Gamma} $ forms a Massey product; 
Hence we clarify partially the Massey product structure of some Brieskorn manifolds. 
Here remark that there are a few methods to compute Massey products with $\Z /p$-coefficients, 
in a comparison with those with $\mathbb{Q} $-coefficients viewed from rational homotopy theory.

%\begin{rem}\label{torusrem23} 
Finally, we comment on the interesting result in Theorem \ref{torus} (ii). 
For finite nilpotent groups $G$, the Massey products in $H^3_{\gr }(G ;\mathbb{F}_q) $ with $p=2, 3$ are exceptional. 
For example, 
when $q=p^2$, the group $G_X$ is isomorphic to the group $``P(3)"$ in \cite{Leary}. 
See \cite[Theorems 6 and 7]{Leary} for exceptional phenomenon of the cohomology ring $H_{\gr}^*(G_X; \Fp)$ with $p=2, 3$.
%and is detected by restrictions to proper subgroups of $G_X$ (see the introduction in \cite{Leary}).
% For example, when $\ww =-1$ and $(n,m,\ell)=(2,3,2)$, 
% the cocycle invariant is non-trivial and the branched covering space $ \widehat{C}_{L}^{2}$ is the lens space $L(3,1)$, 
% while the Massey products of $L(p,1)$ with $p>3$ are trivial. 
%\end{rem}

%We remark the conditions of $n,m$ and $l$ in Corollary \ref{torus314} , following \cite{Milnor}. 

\subsection{Further examples in the case $\ww = -1$}\label{www}

We change our focus to other knots. 
However, it is not so easy to calculate the cocycle invariant $I_{ \Gamma(\qq)_e}(K)$ of knots, although it is elementary.
% We now introduce some values of some knots, without the detailed calculations. 

We now consider the simplest case $\ww =-1$; hence the quandle $X$ is of type $2$. 
Furthermore, note that, for any Mochizuki quadruple $\mathfrak{q}= (q_1,q_2,q_3,q_4)$, the associated 3-cocycle forms $U_1^{q_1}U_2^{q_2+q_3}U_3^{q_4}$ by definition;
it is not hard to compute the cocycle invariant. 
However, for many knots whose colorings satisfy $\col_X(D) \cong (\F)^2$, the invariants are frequently of the form $q^2 \sum_{a \in \F} a^{q_1 +q_2 +q_3+q_4}$ up to constant multiples in computer experiments.
In order to avoid the case $\col_X(D) \cong (\F)^2$,
recall the bijection \eqref{Inoue}. Accordingly we shall deal with some knots having non-trivial second Alexander polynomials as follows: 
\begin{exa}\label{ex12}Let $\ww =-1$.
The knots $K$ in Table \ref{G_c} are those whose crossing numbers are $< 11$ satisfying $\col_X(D) \cong (\F)^3 $ with $p>3$.
We only list results of the invariants without the proofs. 
Here note that, according to Theorem \ref{thm3} and Proposition \ref{lll}, the 
cocycle invariant stems from triple Massey products of double branched covering spaces. 
Refer to the tables in \cite[Appendix F]{Kaw} for the correspondences between knots $K$ and double coverings of $S^3$ branched over $K$.

\begin{table}[htbp]

\vskip 0.3937pc

\large \hspace*{5.377pc}
\begin{tabular}{|c|c|c|}
\hline $ K$ & $ p$ & $ \ \ \ I_{\Gamma(\qq)_1}(K) \ \ $ \\
\hline\hline $9_{40}$ & $5$ & $ \ \ \mathcal{G}(\qq; 1,5)\ \ $ \\
\hline $9_{41}$ & $7$ & $\ \ \mathcal{G}(\qq; 3,4) \ \ $\\
\hline $9_{49}$ & $5$ & $\ \ \mathcal{G}(\qq; 3,4) \ \ $ \\\hline 

\end{tabular} 

\vskip -5.137pc

\large \hspace*{19.377pc}
\begin{tabular}{|c|c|c|}
\hline $ K $ & $ p$ & $ \ \ I_{\Gamma(\qq)_1}(K) \ \ $ \\
\hline \hline $10_{103}$ & $5$ & $\ \ \mathcal{G}(\qq; 2,1) \ \ $ \\
\hline $10_{123}$ & $11$ & $\ \ q^4 $ \\
\hline $10_{155}$ & $5$ & $ \ \ \mathcal{G}_{155}(\qq )\ \ $
\\
\hline $10_{157}$ & $7$ & $ \ \ \mathcal{G}(\qq; 1,5) \ \ $ \\\hline 

\end{tabular} 
\vskip -0.2137pc
\caption{The values of $ I_{\Gamma(\qq)_1}(K)$. }

\label{G_c}
\end{table}

%\vskip -0.37pc

\noindent
Here, $q \in \Z $ and $\qq \in \mathcal{Q}_{q,\ww}$ are arbitrary, and, for $n,m \in \Z$, the symbols $ \mathcal{G}(\qq; n,m) $ and $ \mathcal{G}_{155}(\qq )$
are polynomials expressed by 
\[ \mathcal{G}(\qq; n,m) := q^2\sum_{a,b \in \F} 
1_{\Z} \bigl\{ n ( a^{q_1+q_2 }b^{q_3 +q_4} +a^{q_3+q_4}b^{q_1 +q_2}+ a^{q_1+q_3}b^{q_2 +q_4} +a^{q_2+q_4}b^{q_1 +q_3}) \]
\vskip -1.837pc
\[\ \ \ \ \ \ \ \ \ \ \ \ \ \ \ \ \ \ \ \ \ \ \ \ \ \ \ \ \ \ \ \ \ \ \ + m ( a^{q_1+q_4}b^{q_2 +q_3} + a^{q_2+q_3}b^{q_1 +q_4}) \bigl\} \in \Z[\F], \]
\[ \mathcal{G}_{155}(\qq) := q^2 \!\! \sum_{a,b \in \F}\!\! 1_{\Z} \bigl\{ 4( a^{q_1+q_2+q_3+q_4}+ a^{q_1}b^{q_2+q_3+q_4} ) + ( a^{q_1+q_2+q_3}b^{q_4} + a^{q_1+q_2}b^{q_3 +q_4} ) \]
\[ \ \ \ \ \ \ \ \ \ \ \ \ \ \ \ \ \ \ \ \ \ \ \ \ +2 ( a^{q_1+q_2+q_4 }b^{q_2} +a^{q_1+q_2+q_4}b^{q_3}+ a^{q_1+q_3}b^{q_2 +q_4} +a^{q_2+q_4}b^{q_1 +q_3} ) \bigr\} \in \Z[\F]. \]

\end{exa}

\subsection{Proof of Theorem \ref{torus}}\label{proprof}
For the proof, 
we first recall a slight reduction of the cocycle invariant by \cite[Theorem 4.3]{IK}, that is, we may consider only shadow colorings of the forms $ \sh = (\CC;0)$. More precisely, 
\begin{equation}\label{useful} I_{\psi}(L) = q \cdot \sum_{\CC \in \col_X(D)}1_{\Z } \{ \langle \psi, [(\CC;0)] \rangle \} \in \Z[A].
\end{equation}

%Let us employ some notation in \cite[\S 6]{AS}, where the cocycle invariants using the 3-cocycle $E_0(1,p)$ of $T(m,n)$ were calculated. 
We establish terminologies on the torus knot $T(m,n)$.
Regard $T(m,n)$ as the closure of a braid $ \Delta^m $, where $\Delta := \sigma_{n-1} \cdots \sigma_{1} \in B_n$. % (see, e.g., Figure 7 in \cite{AS}).
Let $\alpha_1, \dots, \alpha_n$ be the top arcs of $\Delta^m$.
For $1 \leq i \leq m$, 
We let $x_{i,1}, \dots , x_{i,n-1}$ be the crossings in the $i$-th $\Delta$; see Figure \ref{fig.color6}. 

\vskip -0.39937pc
\begin{figure}[htpb]

\ \ \ \ \ \ \ \ \ \ \ \ \ \ \ \ \ \ \ \ \ \ \ \ \ \ \ \ \ \ \ \ \ \ \ \ \ \ \ \ \ \ \ \ \ \ \ \ \ \ \ \ \ \ \ \ \ \ \ \ \ \ \ \ 
\includegraphics*[width=19.3cm,height=4.6198cm]{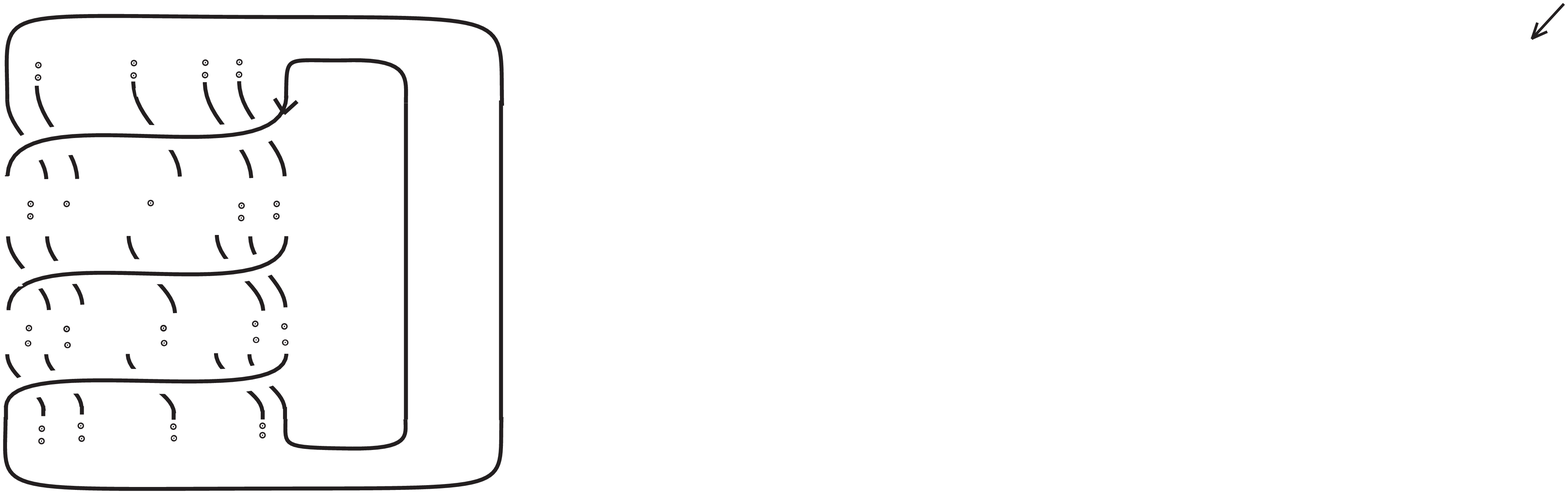} 
\begin{center}

\vskip -1.89107pc
\begin{picture}(100,20)

\put(-21,116){\Large $\cdots $}
\put(11,116){\Large $\cdots $}

\put(-5,89){\Large $\cdots  $}
\put(27,89){\Large $\cdots  $}

\put(-5,55){\Large $\cdots  $}
\put(27,55){\Large $\cdots  $}

\put(-17,38){\Large $\cdots $}
\put(15,38){\Large $\cdots $}

\put(-2,21){\Large $\cdots  $}
\put(29,21){\Large $\cdots  $}

\put(-49,109){\large $\alpha_1 $}

%\put(-24,102){\large $\alpha_2 $}
\put(-22,109){\large $\alpha_2 $}
\put(15,109){\large $\alpha_j $}
% \put(32,102){\large $\alpha_i $}
\put(63,114){\large $\alpha_n $}

\put(-49.6,64){\large $x_{i,1 } $}
\put(-17,69.6){\large $x_{i,2 } $}
\put(14.05,69.6){\large $x_{i,j} $}

\put(63,66){\large $x_{i,n-1 } $}

\end{picture}
\end{center}
\vskip -2.049937pc
\caption{The arcs $\alpha_j$ and crossing points $x_{i,j }$ on the diagram of the torus knot.
\label{fig.color6}}
\end{figure}

\begin{proof}[Proof of Theorem \ref{torus}]
Although Asami and Kuga \cite{AK} formulated explicitly $X$-colorings of $T(m,n)$, 
we will give a reformulation of them appropriate to the 3-cocycle $\Gamma (\mathfrak{q})_e $.
If given an $X$-coloring $C$ of $T(m,n)$, we define $ a_j := \CC(\alpha_j)$, and put a vector $\textbf{a}=(a_1, \dots, a_n) \in (\F)^n$; 
Notice that it satisfies the equation $\textbf{a} = \textbf{a} P^m $, where $P$ is given by 
a companion matrix 
\[ P:= 
\left(
\begin{array}{cccccc}
0 & \ww & 0& \cdots& 0 & 0 \\
0 & 0 & \ww & \cdots & 0& 0\\
0 & 0 & 0& \ddots & 0& 0\\
\vdots & \vdots & \vdots & \ddots & \ddots & \vdots \\
0 & 0 & 0&\cdots & 0 & \ww \\
1 & 1- \ww & 1-\ww & \cdots &1- \ww & 1-\ww
\end{array}
\right) \in \mathrm{Mat}(n \times n ; \F).\]
Remark that the characteristic polynomial is $ (\lambda -1)(\lambda^n - \ww^n )/(\lambda -\ww)$, and that the roots are $ \lambda = \zeta^k \ww$ and $1$, 
where $1 \leq k < n$ and $\zeta$ means an $n$-th primitive root of unity in the algebraic closure $\overline{\mathbb{F}}_p.$
Therefore, the proof comes down to the following two cases: 

\vskip 0.7pc

\noindent 
\textbf{Case I} \ \ \ \!\!\!\! $\ww^n \neq 1$. Namely, the roots are distinct.

\noindent
\textbf{Case II} \ $\ww^n =1$. Then, $\lambda = 1$ is a unique double root of the characteristic polynomial.

\vskip 0.7pc

We will calculate the weights coming from such $X$-colorings case by case.
While the statement (i) will be derived from Case I,
those (ii) and (iii) will come from Case II.

\noindent
({\bf Case I}) \ Let $\ww^n \neq 1 $. We will study the solusions of $\textbf{a} = \textbf{a} P^m $. 
We easily see that, if $(\zeta^{-k} \ww)^m = 1 $ for some $k$, 
then the solution is of the form 
$$ a_{j+1} = a \bigl(( 1- \zeta^{kj})/ (1 - \zeta )\bigr) + a \bigl( \zeta^{kj } / ( 1-\ww )\bigr) + \delta$$
for some $a, \delta \in \overline{\mathbb{F}}_p $; 
conversely, if the equation $\textbf{a} = \textbf{a} P^m $ has a non-trivial solution, then there is a unique $k$ satisfying $(\zeta^{-k} \ww)^m = 1 $ and $0 < k <n$.
It is further verified that such a solution gives rise to an $X$-coloring $\CC $ if and only if $a,\delta, \zeta$ are contained in $\F$. 
In assumary, we may assume that $a,\delta,\zeta \in \F$ and $(\zeta^{-1} \ww )^m=1$ with $\zeta \neq \ww$.
Indeed this assumption justifies a shadow coloring $\sh$ of the form $(\CC; 0)$.

\begin{rem}\label{remos}
We give a remark on this assumption. Notice that, for $s \in \Z$, two equalities $\ww^m = \zeta^m$ and $\ww^s=1$ imply 
$ \zeta^{ms } =1 $ and, hence, $\zeta^s=1$, since $m$ and $n$ are coprime. 
In particular, considering special cases of $s= q_1 +q_3$ and $s=q_2 +q_4$, 
we have $\zeta^{q_1+q_3}=\zeta^{q_2 +q_4}= 1 $.
Similarly we notice that, if $\ww^{q_1+q_2}=1$, then $ \zeta^{q_1+q_2}=1$.
\end{rem}

We will present the weights of $[\sh]=[(\CC;0)]$, where $\CC$ is the $X$-coloring as the solusion mentioned above.
We then can easily check the color of every regions in the link-diagram.
After a tedious calculation, the weight of $x_{i,j}$ is consequently given by 
\[ \Bigl( \ a \zeta^{-i} \ww^i \frac{1- \zeta^{j-1}}{1- \zeta} + (1- \ww^{j-1}) \delta, \ a \zeta^{-i} 
\ww^i ( \frac{1- \zeta^{j-1}}{1- \zeta} + \frac{\zeta^{j-1}}{1-\ww} ) + \delta ,\ \frac{a \zeta^{-i-1} \ww^{i+1} }{1-\ww} + \delta \ \Bigr) \in C_3^Q(X).\]

We next compute the pairing $\langle \Gamma (\mathfrak{q})_e, [\sh] \rangle \in \F $ in turn. To begin with the case $e=1$, 
recalling $\Gamma (\mathfrak{q})_e = U_1^{q_1}U_2^{q_2 +q_3}U_3^{q_4}= (x_1-x_2)^{q_1}(x_2-x_3)^{q_2 +q_3}x_3^{q_4}$, we describe the paring as
\[ \sum_{i\leq m , \ j \leq n-1 } ( \frac{ a \zeta^{-i} \ww^i \zeta^{j-1}  }{1-\ww} - \ww^{j-1} \delta \bigr)^{q_1} 
\bigl( \frac{ a \zeta^{-i-1} \ww^{i} (\ww - \zeta) ( \zeta^{j} - 1 )}{(1-\ww)(1- \zeta )} \bigr) ^{q_2+q_3}
\bigl( \frac{a \zeta^{-i-1} \ww^{i+1} }{1-\ww} + \delta \bigr)^{q_4 } \in \F . \]
Here we note $\sum_{i=1}^{m} (\zeta^{-1} \ww)^{si}= 0$ unless $\zeta^{-s} \ww^s= 1$. 
Therefore, several terms in this formula vanish by Remark \ref{remos} above. 
It is easily seen that the non-vanishing term in $\langle \Gamma (\mathfrak{q})_1, [\sh] \rangle$ is 
\begin{equation}\label{116g} 
\frac{ a^{q_1 +q_2+q_3 +q_4} (\zeta -\ww )^{q_2+q_3} \ww^{q_4}
}{(1-\ww)^{q_1 +q_2+q_3+q_4}(1-\zeta)^{q_2+q_3} }
\sum_{i \leq m, \ j \leq n-1} (\zeta^{-1}\ww)^{i (q_1 +q_2+q_3 +q_4) } \zeta^{j q_1} (1 -\zeta^j )^{ q_2+ q_3 } \in \F . \end{equation} 
Here, by Remark \ref{remos} again, we notice two equalities
$$ (\zeta^{-1}\ww) ^{ q_1 +q_2+q_3 +q_4 }=1, \ \ \ \ \zeta^{j q_1} (1 -\zeta^j)^{ q_2+ q_3 }= \zeta^{j q_1} + \zeta^{j q_2}- 2. $$
Therefore, since $\sum_{j=1}^{n-1} \zeta^{j q_1} =\sum_{j=1}^{n-1} \zeta^{j q_2} =-1 $, the sum in the formula \eqref{116g} equals $-2nm$. 
By \eqref{useful}, we hence obtain the required formula \eqref{wwneq1}.

Further, by Remark \ref{chuui}, the same calculations hold for the cases $ 3\leq e \leq 5$.

Next, we deal with $e=2$.
For the shadow coloring $\sh= (\CC;0)$, we claim $\langle \Gamma (\mathfrak{q})_2, [\sh] \rangle =0$.
To see this, by a similar calculation to \eqref{116g}, we reduce 
the paring $\langle \Gamma (\mathfrak{q})_2, [\sh] \rangle$ to 
$$\langle \Gamma (\mathfrak{q})_2, [\sh] \rangle= -2nm a^{q_1 +q_2+q_3 +q_4} \cdot \mathcal{A}_{\qq},$$
where $\mathcal{A}_{\qq}$ is temporarily defined by the formula
\[ \frac{ (\zeta -\ww )^{q_2+q_3} \ww^{q_4}
}{ (1-\zeta)^{q_2+q_3} } -\frac{ (\zeta -\ww )^{q_1+q_4} \ww^{q_3}
}{ (1-\zeta)^{q_1+q_4} } + \frac{1-\ww^{q_1 + q_2}}{1-\ww^{q_2} } \bigl( \frac{ (\zeta -\ww )^{q_2} \ww^{q_3+ q_4}
}{ (1-\zeta)^{q_2} } -\frac{ (\zeta -\ww )^{q_4} \ww^{q_3}
}{ (1-\zeta)^{q_4} } \bigr) . \]
We now assert that the last term in this formula $\mathcal{A}_{\qq}$ is zero.
Indeed, noting $(1-\zeta)^{-q_4}= \zeta^{q_2} (1-\zeta)^{-q_4}\zeta^{-q_2} =\zeta^{q_2} (\zeta-1)^{-q_2}$ by Remark \ref{remos}, 
we easily have
\[ \frac{ (\zeta -\ww )^{q_2} \ww^{q_3+ q_4}
}{ (1-\zeta)^{q_2} } -\frac{ (\zeta -\ww )^{q_4} \ww^{q_3}
}{ (1-\zeta)^{q_4} } = \frac{ (\zeta -\ww )^{q_2} \ww^{q_3+ q_4} +(\zeta -\ww )^{q_4} \zeta^{q_2}\ww^{q_3}
}{ (1-\zeta)^{q_2} }= 0. \]
Similarly we easily see an equality $(1-\zeta)^{-q_1-q_4}= \zeta^{q_2+q_3} (1-\zeta)^{-q_2-q_3}$; 
therefore the first and second terms in $\mathcal{A}_{\qq }$ are canceled. Hence $\mathcal{A}_{\qq}=0$ as claimed. 
In conclusion, the cocycle invariants using $\Gamma(\qq)_2$ are trivial as desired.

\

\noindent
(\textbf{Case II}) We next consider another case of $\ww^n =1$. 
Notice that the matrix $P - E_n$ is of rank $n-1$. 
Hence, if the above equation $\textbf{a} = \textbf{a} P^m $ has a non-trivial solution, 
then $m$ must be divisible by $p$ (consider the Jordan-block of $P$). 
For such an $m$, we can verify that the solution is 
of the form $a_j = a \ww -a \ww^j +\delta$ for some $a, \delta \in \F $, which provides an $X$-coloring $\CC$. 
Put a shadow coloring of the form $\sh = (\CC; 0)$. 
The weight of the crossing $x_{i,j}$ is then given by 
\[ \Bigl( a \bigl(1- j )(1 -\ww) \ww^{j-1} + a i(\ww -1)(1- \ww^{j-1})+ (a+\delta)(1-\ww^{j-1}), a(1-\ww^j+ i\ww -i )+ \delta, \ ai(\ww -1) + \delta \Bigr). \]
\ Let us calculate the pairings $\langle \Gamma (\mathfrak{q})_e, [\sh] \rangle$. To begin, when $e=1$, the $\langle \Gamma (\mathfrak{q})_1, [\sh] \rangle$ equals
\begin{equation}\label{11146g}\!\!\!\!\!\!\!\!\!\!\! \sum_{\ \ \ i \leq m , \ j \leq n-1 } \!\!\!\! \bigl( (aj (\ww -1)-ai (\ww-1) -\delta) \ww^{j-1} \bigr)^{q_1}
\bigl(a -a \ww^{j} \bigr)^{q_2+q_3} \bigl( a i ( \ww -1) + \delta \bigr)^{q_4 }. \end{equation}
We here consider the sum on $i$. 
However, we notice that $\sum_{ i \leq m}i^{q_1 +q_4}= \sum_{ i \leq m} i^{2}= m(m+1)(2m+1)/6$. 
Hence, since $m$ is divisible by $p$, the pairing vanishes unless $p=2, 3$. 

Similarly, we can see that, in other cases of $e$, the pairings are zero unless $p=2, 3$. 
We therefore may devote to the cases $p=2,3 $ hereafter.

First, assume $p=3$ and $e=1$. 
Note that the non-vanishing term in \eqref{11146g} is only the coefficients of $\sum i^{q_1+q_4}$, and that $ \sum_{ i \leq m}i^{q_1 +q_4} = -m/3$. Then the pairing \eqref{11146g} is reduced to be 
\[ a^{q_1 +q_2 +q_3 +q_4}(1- \ww )^{q_1 +q_4} \!\!\! \sum_{1 \leq j\leq n-1}\!\!\! \ww^{q_1(j-1)} (1 - \ww^{j})^{q_2+q_3} \!\! \sum_{1 \leq i \leq m} i^{q_1+q_4} = \frac{mn}{3}a^{q_1 +q_2 +q_3 +q_4}(1- \ww )^{q_3 +q_4}, \] 
where $\sum \ww^{q_1(j-1)} (1 - \ww^{j})^{q_2+q_3} =2n \ww^{-q_1}$ in this equality follows from $\ww^n=1$.
Hence, by running over all shadow colorings, we obtain the required formula \eqref{xwwedqual1}.
Similarly, when $p=2$ and $e=1$, a calculation using Lemma \ref{odddegree2}(I) below can show the formula \eqref{xwwedqual1}.

Furthermore, the same calculation holds for the cases $ 3\leq e \leq 5$ and $p=2,3$.
Actually, it is done by changing the quadruple ($q_1,q_2,q_3,q_4$) in the previous calculation in Case 1, as a routine reason for the cases. 

At last, it is enough for the proof to work out the remaining case $e=2$ and $p=2,\ 3$. 
By Lemma \ref{odddegree2} (II) below and the definition of $\Gamma (\mathfrak{q})_2$, the pairing is reduced to 
\begin{equation}\label{useful2rw}
\langle \Gamma (\mathfrak{q})_2, [\sh] \rangle = \langle U_1^{q_1} U_2^{q_2+q_3} U_3^{q_4} , [\sh] \rangle - \langle U_1^{q_2} U_2^{q_1+q_4} U_3^{q_3} , [\sh] \rangle.
\end{equation}
We claim that if $p=3$, $\langle \Gamma (\mathfrak{q})_2, [\sh] \rangle =0$.
The first term is reduced to $2mn a^{q_1 +q_2 +q_3 +q_4}(1- \ww )^{q_3 +q_4}/3$, by a similar calculation to \eqref{xwwedqual1}.
The second term is obtained by changing the indices $(1,2,3,4)$ in the first term to $(2,1,4,3)$.
Hence the pairing $\langle \Gamma (\mathfrak{q})_2, [\sh] \rangle$ vanishes. 

To complete the proofs, we let $p=2$. 
The explicit formula of the first term in \eqref{useful2rw} follows from Lemma \ref{odddegree2} (III) below.
Furthermore, by the previous change of the indices, we know the second term. 
In summary, we conclude the desired formula in (iii). 
\end{proof}

The following lemma used in the above proof can be 
obtained from the definitions and elementary calculations, although 
they are a little complicated.

\begin{lem}\label{odddegree2}
Let $\sh =(\CC ;0)$ be the shadow coloring in \textbf{\rm \textbf{Case II}} as above. 
%Let $U_1 = x_1 -x_2$, $U_2 = x_2 -x_3$ and $U_3 = x_3$ explained in \S \ref{reMoc}. 
\begin{enumerate}[(I)]
\item If $p=2$ and $\ww^{q_1+q_2}=1$, then the pairing $\langle U_1^{q_1} U_2^{q_2+q_3} U_3^{q_4},[\sh] \rangle$ is equal to $(1+ \ww)^{q_3+ q_4} a^{q_1 + q_2+q_3 +q_4} mn/2$. 
\item If $\ww^{q_1+q_2} \neq 1$, and if $p=2$ or $3$, then 
$ \langle U_1^{q_1} U_2^{q_2} U_3^{q_3+q_4} - U_1^{q_1 +q_2} U_2^{q_4} U_3^{q_3} , [\sh]\rangle =0.$
\item If $p=2$ and $\ww^{q_1+q_2} \neq 1$, then $\langle U_1^{q_1 } U_2^{q_2+q_3} U_3^{q_4} , [\sh]\rangle$ is equal to
$$ \frac{mn}{2} \Bigl(a^{q_2+q_3} \bigl( (1+\ww^{q_1}) a^{q_1}\delta^{q_4} + (1+\ww^{q_4} )a^{q_4}\delta^{q_1} \bigr) + \displaystyle{\frac{1+\ww^{-q_1}+\ww^{-q_2}+\ww^{q_1+q_2}}{1+\ww^{q_1 +q_2}} a^{q_1+q_2+q_3+q_4} }\Bigr) . $$
\end{enumerate}
\end{lem}

\subsection*{Acknowledgment}
The author expresses his gratitude to Tomotada Ohtsuki and Michihisa Wakui for valuable comments on group cohomologies and 3-manifolds. 
He is particularly grateful to Yuichi Kabaya for useful discussions and making several suggestions for improvement. 
%This work is supported by ``GCOE program" in RIMS.

\vskip 1pc

\normalsize

% Research Institute for Mathematical Sciences, Kyoto University, Sakyo-ku, Kyoto, 606-8502, Japan

Faculty of Mathematics, Kyushu University, 744, Motooka, Nishi-ku, Fukuoka, 819-0395, Japan

\

E-mail address: {\tt nosaka@math.kyushu-u.ac.jp}


\begin{thebibliography}{99}

\normalsize
\ifx\undefined\bysame
\newcommand{\bysame}{\leavevmode\hbox to3em{\hrulefill}\,}
\fi


\bibitem[AK]{AK} S. Asami, K. Kuga, {\it Colorings of torus knots and their twist-spuns by Alexander quandles over finite fields}. J. Knot Theory Ramifications {\bf 18} (2009), no. 9, 1259--1270.



\bibitem[Bro]{Bro}
K. S. Brown, 
{\it Cohomology of Groups}, 
Graduate Texts in Mathematics, {\bf 87}, Springer-Verlag, New York, 1994.


%\bibitem[CEGS]{CEGS} 
%J. S. Carter, J. S. Elhamdadi, M. Gra\~na, M. Saito,
%{\it Cocycle knot invariants from quandle modules and 
%generalized quandle homology}, 
%Osaka J. Math. {\bf 42} (2005), 499--541.

\bibitem[CJKLS]{CJKLS} 
J. S. Carter, D. Jelsovsky, S. Kamada, L. Langford, M. Saito, 
{\it Quandle cohomology and state-sum invariants of knotted curves and surfaces}, 
Trans. Amer. Math. Soc. {\bf 355} (2003) 3947--3989. 


\bibitem[CKS]{CKS}
J. S. Carter, S. Kamada, M. Saito,
{\it Geometric interpretations of quandle homology}, 
J. Knot Theory Ramifications {\bf 10} (2001) 345--386.

\bibitem[Cla]{Cla} 
F.J.-B.J. Clauwens, 
{\it The adjoint group of an Alexander quandle}, arXiv:math/1011.1587. 


\bibitem[DW]{DW}
R. Dijkgraaf, E. Witten,
{\it Topological gauge theories and group cohomology},
Comm. Math. Phys. {\bf 129} (1990) 393--429.

\bibitem[FRS1]{FRS1} 
R. Fenn, C. Rourke, B. Sanderson,
{\it Trunks and classifying spaces}, 
Appl. Categ. Structures {\bf 3} (1995) 321--356.

\bibitem[FRS2]{FRS2} 
\bysame, 
{\it The rack space}, Trans. Amer. Math. Soc. {\bf 359} (2007), no. 2, 701--740.


%\bibitem[E1]{Eis1} M. Eisermann, 
%{\it Knot colouring polynomials},
%Pacific Journal of Mathematics {\bf 231} (2007) 305--336.


\bibitem[Eis]{Eis2} 
M. Eisermann, 
{\it Quandle coverings and their Galois correspondence}, arXiv:math/0612459v3


%\bibitem[Hat]{H} E. Hatakenaka,
%{\it Invariants of 3-manifolds derived from covering presentations}, in Math. Proc. Camb. Phil. Soc. {\bf 149} (2010) 263--296.


%\bibitem[HN]{HN} \bysame, T. Nosaka, {\it Some topological aspects of 4-fold symmetric quandle invariants of 3-manifolds}, to appear Internat. J. Math. 


\bibitem[Ino]{Ino} 
A. Inoue, {\it Quandle homomorphisms of knot quandles to Alexander quandles}, J. Knot Theory Ramifications {\bf 10} (2001) 813--821. 

\bibitem[IK]{IK} 
\bysame, Y. Kabaya, 
{\it Quandle homology and complex volume}, arXiv:math/1012.2923.


\bibitem[Joy]{Joy} 
D. Joyce, 
{\it A classifying invariant of knots, the knot quandle},
J. Pure Appl. Algebra. {\bf 23} (1982) 37--65.



\bibitem[Kab]{K} Y. Kabaya, {\it Cyclic branched coverings of knots and quandle homology,} Pacific Journal of Mathematics, {\bf 259} (2012), No. 2, 315--347.


\bibitem[Kaw]{Kaw} A. Kawauchi ed., {\it A survey of knot theory}, Birkh\"{a}user, Basel, (1996)

\bibitem[Kra]{Kra}
D. Kraines, {\it Massey higher products}, Trans. AMS, {\bf 124} (1966) 431--449.

\bibitem[Le]{Leary}
I. J. Leary, {\it The mod-$p$ cohomology rings of some $p$-groups}, Math. Soc. Cambridge Philos. Soc. {\bf 112} (1992) 63--75.


\bibitem[Man]{Man}
J. Mandemaker, {\it Various topics in rack and quandle homology}, Master thesis in Radboud University, Nijmegen, 2009 


\bibitem[Moc]{Moc2} T. Mochizuki,
{\it The 3-cocycles of the Alexander quandles $\mathbb{F}_q [T]/(T- \omega)$}, Algebraic and Geometric Topology. {\bf 5} (2005) 183--205.


\bibitem[Mil]{Milnor} J. W. Milnor, 
{\it On the 3-dimensional Brieskorn manifolds $M(p, q, r)$, 
Knots, groups and
3-manifolds}, Ann. of Math. Studies {\bf 84}, Princeton Univ. Press, 1975, 175--225.




\bibitem[No1]{Nosa}
T. Nosaka,
{\it On quandle homology groups of Alexander quandles of prime order}, to appear Trans. Amer. Math. Soc. %preprint, arXiv:1103.3839.


\bibitem[No2]{Nos3}
\bysame,
{\it Quandle cocycles from invariant theory}, preprint
 


\bibitem[No3]{Nos2}
\bysame,
{\it Topological interpretation of link invariants from finite quandles}, preprint


% \bibitem[Rol]{Rolfsen}
% D. Rolfsen, {\it Knots and links}, Math. Lecture Series, 7,
% Publish or Perish, Inc., Houston, Texas, 1990, Second printing, with corrections.


\bibitem[Wa]{Wakui}
M. Wakui,
{\it On Dijkgraaf-Witten invariant for 3-manifolds},
Osaka J. Math. {\bf 29} (1992) 675--696.

\bibitem[Wee]{Wee} J. Weeks, {\it Computation of Hyperbolic Structures in Knot Theory}, Handbook of Knot Theory.



\end{thebibliography}
\end{document}